\documentclass{amsart}

\usepackage{times, amsmath,amssymb,amsthm, amsxtra}
\usepackage{graphicx}
\usepackage{a4wide}
\usepackage[all]{xy}
\usepackage{mathrsfs}  

\usepackage{color}
\definecolor{Chocolat}{rgb}{0.36, 0.2, 0.09}
\definecolor{BleuTresFonce}{rgb}{0.215, 0.215, 0.36}

\usepackage[colorlinks,final,hyperindex]{hyperref}
\hypersetup{citecolor=BleuTresFonce, linkcolor=Chocolat}

\DeclareMathAlphabet{\mathbbold}{U}{bbold}{m}{n}

\DeclareSymbolFont{rsfscript}{OMS}{rsfs}{m}{n}
\DeclareSymbolFontAlphabet{\mathrsfs}{rsfscript}

\DeclareFontFamily{OMS}{rsfs}{\skewchar\font'177}
\DeclareFontShape{OMS}{rsfs}{m}{n}{%
      <5> rsfs5
      <6> <7> rsfs7
      <8> <9> <10> rsfs10
      <10.95> <12> <14.4> <17.28> <20.74> <24.88> rsfs10
      }{}

\newcommand{\ba}{\bar{\alpha}}
\newcommand{\A}{\mathcal{A}}
\newcommand{\LL}{\mathrm{L}}
\newcommand{\R}{\mathrm{R}}
\newcommand{\oPa}{\overline{\P}^{\ac}}
\newcommand{\cc}{\circledcirc}
\newcommand{\lam}{\lambda}
\newcommand{\pullback}{\mbox{\Large{$\lrcorner$}}}
\newcommand{\BCH}{\mathrm{BCH}}
\newcommand{\MC}{\mathrm{MC}}
\newcommand{\at}{{\alpha}}

\newcommand{\ad}{\mathrm{ad}}
\newcommand{\epi}{\twoheadrightarrow}

\newcommand{\mono}{\rightarrowtail}

\newcommand{\NN}{\mathbb{N}}
\newcommand{\Sy}{\mathbb{S}}
\def\KK{\mathbb{K}}
\newcommand{\ac}{\scriptstyle \text{\rm !`}}
\DeclareMathOperator{\D}{D}
\def\g{\mathfrak{g}}
\def\a{\mathfrak{a}}
\def\TTT{\mathcal{T}}

\def\qi{\xrightarrow{\sim}}

\def\P{\mathcal{P}}
\def\hoP{\mathrm{ho}\, \mathcal{P}}
\def\I{\mathrm{I}}

\DeclareMathOperator{\id}{id}
\DeclareMathOperator{\End}{End}
\DeclareMathOperator{\Hom}{Hom}

\makeatletter

\theoremstyle{plain}

\newtheorem {theorem}{Theorem}
\newtheorem*{theo3}{Theorem 3}
\newtheorem*{theo2}{Theorem 2}
\newtheorem*{theo4}{Theorem 4}
\newtheorem {lemma}{Lemma}
\newtheorem {corollary}{Corollary}

\newtheorem {proposition}{Proposition}

\theoremstyle{definition}

\newtheorem {definition}{Definition}
\newtheorem*{example}{\sc Example}
\newtheorem*{remark}{\sc Remark}
\newtheorem*{remarks}{\sc Remarks}
\newtheorem*{notation}{\sc Notation}



\thanks{S.S. was supported by the Netherlands Organisation for Scientific Research. B.V. was supported by the ANR SAT grant.}

\begin{document}

\title[Pre-Lie deformation theory]{Pre-Lie deformation theory}

\date{\today}

\author{Vladimir Dotsenko}
\address{School of Mathematics, Trinity College, Dublin 2, Ireland}
\email{vdots@maths.tcd.ie}

\author{Sergey Shadrin}
\address{Korteweg-de Vries Institute for Mathematics, University of Amsterdam, P. O. Box 94248, 1090 GE Amsterdam, The Netherlands}
\email{s.shadrin@uva.nl}

\author{Bruno Vallette}
\address{Laboratoire Analyse, G\'eom\'etrie et Applications, Universit\'e Paris 13, Sorbonne Paris Cit\'e, CNRS, UMR 7539, 93430 Villetaneuse, France.}
\email{vallette@math.univ-paris13.fr}

\begin{abstract}
In this paper, we develop the deformation theory controlled by pre-Lie algebras; the main tool is a new integration theory for pre-Lie algebras.
The main field of application lies in homotopy algebra structures over a Koszul operad; in this case, we provide a  
 homotopical description of the associated Deligne groupoid. 
This permits us to give a conceptual proof, with complete formulae, of 
 the Homotopy Transfer Theorem by means of gauge action. 
We provide a clear explanation of this latter ubiquitous result: there are two gauge elements whose action on the original structure restrict its inputs and respectively its output to the homotopy equivalent space. 
This implies that a homotopy algebra structure transfers uniformly to a trivial structure on its underlying homology if and only if it is gauge trivial; this is  the ultimate generalization of the $dd^c$-lemma.
\end{abstract}

\maketitle

\setcounter{tocdepth}{1}

\tableofcontents

\section*{Introduction}

The philosophy of deformation theory, which goes back to P. Deligne, 
{M. Gerstenhaber, W.M. Goldman, }
A. Grothendieck, 
{J.J. Millson}, 
M. Schlessinger, J. Stasheff, and many others, says that any deformation problem over a field of characteristic zero can be encoded by a differential graded Lie algebra. A precise mathematical statement, together with a proof of it, has recently been given by J. Lurie \cite{Lurie11} with the help of higher category theory. 

\smallskip

In general, algebras of a given type do not have nice homotopy properties. Instead, one should embed them into a bigger category of \emph{homotopy algebras}, where the original relations are relaxed up to an infinite sequence of homotopies. 
Encoding types of algebras with the notion of an operad, one can apply the Koszul duality of operads to produce the required notion of homotopy algebras and a suitable notion of morphisms between them, called \emph{$\infty$-morphisms}. 

\smallskip

Simultaneously  the Koszul duality of operads also produces a differential graded Lie algebra which encodes homotopy algebra structures, thereby answering the philosophy of deformation theory in this case. This means that the solutions to the Maurer--Cartan equation are in one-to-one correspondence with homotopy algebra structures and that the action of the gauge group, coming from the integration of the Lie algebra into a Lie group, gives a suitable equivalence relation between these elements. These two pieces of data form the \emph{Deligne  groupoid}, which is described by the first result of the present paper. 

\begin{theo3}
The Deligne groupoid coming from  the  differential graded Lie algebra associated to a Koszul operad $\P$ and a chain complex $V$ admits, for objects, the set of $\P_\infty$-algebra structures on $V$ and, for morphisms, the $\infty$-isomorphisms whose first component is the identity. 
\end{theo3}

The proof of this result relies on the following general method. We first notice that this differential graded Lie algebra actually carries the refined algebraic structure of a \emph{pre-Lie algebra}. An algebra of this type is made up of  a binary product whose associator is right-symmetric. So, the induced bracket defined by the commutator of the binary product satisfies the Jacobi identity. This algebraic structure was first discovered on  the Hochschild cochain complex of  associative algebras \cite{Gerstenhaber63} and 
on the space of vector fields on a manifold with a flat connection, where it is called a Vinberg algebra after \cite{Vinberg63}.

\smallskip
In this paper, we develop the integration theory of pre-Lie algebras, that is we construct a group whose tangent space at the unit is the original pre-Lie algebra. 
 Some of the ingredients are not new, since the notions of pre-Lie exponential and pre-Lie logarithm, as known as Magnus expansion, were already present in the seminal work of Agrachev--Gamkrelidze \cite{AgrachevGamkrelidze80} on the solutions of differential equations defined by a flow. 
Recall that one can associate to any pre-Lie product an infinite collection of operations of any arity, called the \emph{symmetric braces}, see \cite{GuinOudom08, LadaMarkl05}. 
We prove  that the product of this gauge group, obtained by integration (exponentiation), is equal to the sum of the symmetric braces. 
\begin{theo2}
Under the pre-Lie exponential map, the gauge group is isomorphic to the group made up of group-like elements equipped with the associative product equal to the sum of all symmetric braces. 
\end{theo2}

{The general construction of the gauge group and its action on Maurer--Cartan elements in differential graded Lie algebras are not very manageable since they rely on the intricate Baker--Campbell-Hausdorff formula. 
In the same way as in the above mentioned theorem, we give a simpler formula for the gauge group action for differential graded pre-Lie algebras. This pre-Lie algebra calculus makes the deformation theory modeled by differential graded pre-Lie algebras easier to work with. Thanks to this more computable calculus, we are able to prove the following results on  the moduli spaces of homotopy algebras.
}

\smallskip

The Homotopy Transfer Theorem, see \cite[Chapter~$10$]{LodayVallette12}, extending the Homological Perturbation Lemma, explains how one can transfer homotopy algebra structures through deformation retracts in a homotopy coherent way. 
Using the  pre-Lie deformation theory, we  shed new light on this theorem. 
With the original algebraic structure and the homotopy datum of the deformation retract, we consider two gauge group elements, both solutions to a certain fixed point equation. Then, the core of the Homotopy Transfer Theorem lies in the following result:

\emph{the action of these two gauge group elements on the original structure produces two $\infty$-isomorphic homotopy algebra structures, which are respectively restricted (inputs) and co-restricted (output) to the homotopy equivalent space}.

\smallskip

 We provide all the formulae for the Homotopy Transfer Theorem, but, more interesting, this approach gives a conceptual explanation for the discrepancy between them: the transferred structure and extension of the inclusion, on the one side, and the extension of the projection and the extension of the homotopy, on the other side. 

\smallskip 

Finally, this study allows us to solve the following problem: find a criterion on the level of a homotopy algebra for the uniform vanishing of the transferred structure on its underlying homology. 

\begin{theo4}
A homotopy algebra structure is homotopy trivial if and only if it is gauge trivial. 
\end{theo4}

Notice that we already stated, proved, and used this result in \cite{DotsenkoShadrinVallette15} but only for the simplest possible example, that of mixed complexes. In that case, this was providing us with the ultimate generalization of the $dd^c$-Lemma of Rational Homotopy Theory due to Deligne--Griffiths--Morgan--Sullivan \cite{DGMS75}.

\subsection*{Layout.} We begin by recalling in Section~\ref{subsec:InfAction} how deformation theory works with differential graded Lie algebras. We make it more precise when it is actually coming from a differential graded  associative algebra (Section~\ref{sec:DefAssoc}); this includes the particular case of homotopy modules over a Koszul algebra (Section~\ref{subsec:ExKoszulAlg}). Then, we develop the deformation theory modeled by  differential graded  pre-Lie algebras {in terms of pre-Lie exponentials and symmetric braces} (Section~\ref{sec:DefPreLie}). It 
 includes the case of homotopy algebras over a Koszul operad (Section~\ref{sec:HomoAlgOp}).  In Section~\ref{sec:HoDiagOP}, we describe the homotopy properties of algebras over an operad by means of diagrams of operads. We treat the  homotopy trivialization in Section~\ref{sec:HoTriv} and the Homotopy Transfer Theorem in Section~\ref{subsec:HTT}. 

\subsection*{Acknowledgements.} {We express our appreciation to Fr\'ed\'eric Chapoton, Dominique Manchon, Tornike Kadeishvili, Martin Markl, and Jim Stasheff  for useful discussions and comments. We would like to thank Jim Stasheff and the referee for careful reading.   }

\subsection*{Convention.} Throughout the paper, we work over a ground field $\KK$ of characteristic $0$. We use the theory of operads as described in \cite{LodayVallette12}.

\section{Deformation theory with Lie algebras}\label{subsec:InfAction}

In this introductory section, we recall the deformation theory controlled by a dg Lie algebra: Maurer--Cartan elements and their equivalence relation. We emphasize the role of the \emph{differential trick}, which consists in making the differential internal, and thus helps to provide neater formulae. \\

Let $\g=(g, {d}, [\, , ])$ be a dg Lie algebra, with homological convention, i.e. $|d|=-1$.
We consider the variety of Maurer--Cartan elements 
$$\boxed{\mathrm{MC}(\g):=\left\{ 
\bar\alpha \in g_{-1}\ Ê| \  d\bar\alpha+ {\textstyle \frac{1}{2}}[\bar\alpha, \bar\alpha]=0
\right\} }\ .$$

Any element $\lambda \in g_0$ induces the following  vector field 
$$ d\lambda + [\lambda, -]= d\lambda + \ad_\lambda \in \Gamma( T\  \mathrm{MC}(\g))\ .$$
Two Maurer--Cartan elements $\bar\alpha, \bar\beta$ are  \emph{equivalent} when there exists an element $\lambda\in g_0$ whose  flow relates $\bar\alpha$ to $\bar\beta$ in finite time. Formulae are easily produced by the following trick.
Let's consider the ``augmented'' dg Lie algebra
$$\g \mono \g^+:=(g\oplus \KK \delta, d,  [\, , ])\ , $$
where $|\delta|=-1$, $d(\delta)=0$,  $[\delta, x]= d(x)$, and $[\delta, \delta]=0$. (One could now forget the differential $d$, since it has been made internal $d=[\delta, \textrm{-}]$ in this new dg Lie algebra.) The affine bijection 
$$\bar\alpha \mapsto {\alpha}:=\delta+\bar\alpha\ ,  $$
for $|\bar\alpha|=-1$, and $x\mapsto x$ otherwise, transforms the Maurer--Cartan equation into the square-zero equation 
$$\boxed{\mathrm{Sq(\g^+)}:=\left\{ 
\at=\delta+\bar\alpha \in g_{-1}\oplus \KK\delta\ | \ 
[\at, \at]=0
\right\}}
\ .$$
Notice that $g_0=(g\oplus \KK \delta)_0$ and that the vector fields now become
$$  [\lambda, -]= \ad_\lambda \in \Gamma( T\  \mathrm{Sq}(\g^+))\ .$$
So the solutions to the differential equation 
$${\gamma}'(t)=\ad_\lambda({\gamma}(t)) $$
are 
$${\gamma}(t)=e^{t\ad_\lambda}({\gamma}_0)\ . $$
If ${\gamma}_0={\alpha}$, and under convergence assumptions, then we define
$$\boxed{{\beta}=e^{\ad_\lambda}({\alpha})}\ .$$
Going back to the Lie algebra $\g$, this gives 
$$\delta + \boxed{\bar\beta} =  e^{\ad_\lambda}(\delta+ \bar\alpha)
=e^{\ad_\lambda}(\delta)+e^{\ad_\lambda}(\bar\alpha)=
 \delta+(e^{\ad_\lambda}-\id)(\delta) +e^{\ad_\lambda}(\bar\alpha)=\delta + 
\boxed{\frac{e^{\ad_\lambda}-\id}{\ad_\lambda}(d\lambda)+e^{\ad_\lambda}(\bar\alpha)}\ .
$$

\section{Deformation theory with associative algebras}\label{sec:DefAssoc}
In this section, we simplify the above-mentioned treatment when the Lie bracket is coming from an associative product. This helps us to introduce the gauge group and to understand why the above equivalence relation is actually coming form a group action. \\

Suppose that the dg Lie algebra $\g=(g, d, [\, , ])$ comes from a unital dg associative algebra $\mathfrak{a}=(A, d, \star, 1)$ under the anti-symmetrized bracket $[x,y]:=x\star y - (-1)^{|x||y|} y\star x$.
Let us further assume the existence of an underlying weight grading, different from the homological degree, 
 satisfying 
$$A\cong\prod_{n=0}^\infty A^{(n)}, \quad 1\in A^{(0)}, \quad A^{(n)}\star A^{(m)} \subset A^{(n+m)}, \quad \text{and} \quad g\cong \prod_{n=1}^\infty A^{(n)}  \quad \text{so} \quad A\cong A^{(0)}\times {g}\ .
 $$
(The following deformation theory works under more general conditions, but the ones given here correspond precisely to  the various examples treated in the present paper.)

Under this assumption, the exponential of elements $\lambda$ 
with trivial weight $0$ component 
is well defined 
$$e^\lambda:=1 +\lambda + \frac{\lambda^{\star 2}}{2!} + \frac{\lambda^{\star 3}}{3!} +\cdots \ . $$
And so the images of the  degree $0$ elements under the exponential map  form a group 
$$\mathfrak{G}:=\left(
\{ e^\lambda, \lambda \in g_0\}=\{1\}\times g_0, 
\star, 1
\right)\ ,  $$
made up of the \emph{group-like elements}, that is the degree $0$ elements whose weight $0$ term is equal to $1$.

The (infinitesimal) adjoint action then reads 
$$\ad_\lambda(\textrm{-})=
\underbrace{(\lambda \star \textrm{-})}_{l_\lambda}-
\underbrace{(\textrm{-} \star \lambda)}_{r_\lambda}
\ .$$ 
Since it comes from   two commuting operators in $\End(g)$, the (full) action  is equal to the conjugation action of the exponential group 
$$e^{\ad_\lambda} (\at)= e^{l_{\lambda}}\circ e^{r_{-\lambda}}(\at)=e^{\lambda}\star \at \star e^{-\lambda}  \ .$$
Notice that, pulling back the product in the exponential group with the logarithm map gives the following group product on  $g_0$: 
$$\boxed{x\cdot y = \ln \left( 
e^x\star e^y 
\right) = x+y+\frac{1}{2}[x,y] +\cdots = \BCH(x,y)}\ ,$$
produced by the Baker--Campbell--Hausdorff formula. 

In general, the dg Lie algebra does not come from a dg associative algebra and the exponential group does not exist. Still, one can consider the group structure, defined by the BCH formula on $g_0$ : 
$$\boxed{\Gamma:=\left(
g_0, x\cdot y := \BCH(x,y), 0
\right)}\ ,$$
under some nilpotence or convergence condition, like the one considered above. This group is called the \emph{gauge group}.

\begin{proposition}[\cite{GoldmanMillson88}]
When the  infinitesimal action of the Lie algebra $g_0$ defined in Section~\ref{subsec:InfAction} 
{converges at} 
 $t=1$, it is  given by the  action of the gauge group $\Gamma$ on the variety of Maurer--Cartan elements under the formula: 
$$ \lambda.\at:=e^{\ad_\lambda}(\at)\ .$$
\end{proposition}

\begin{proof}
The action of $0$ is trivial 
$$0.\at =  e^{\ad_0}(\at)=\id(\at)=\at\ ,$$
and the action respects the group structure 
$$ 
\lambda.(\mu.\at)= e^{\ad_\lambda}(e^{\ad_\mu}(\at))= (e^{\ad_\lambda} \circ e^{\ad_\mu})(\at)=
e^{\ad_{\BCH(\lambda, \mu)}}(\at)=(\lambda\cdot \mu).\at\quad  .
$$
\end{proof}

In this case, one considers the \emph{Deligne groupoid} associated to the dg Lie algebra $\g$: 
$$\boxed{\mathsf{Deligne}(\g):=\left(\mathrm{MC}(\g), \Gamma
\right)} \ ,$$
whose objets are Maurer--Cartan elements of $\g$ and whose morphisms are given by the action of the gauge group $\Gamma$. 

\section{Example : homotopy modules over a Koszul algebra}\label{subsec:ExKoszulAlg}
In this section, we apply the general deformation theory controlled by a dg associative algebra to the example of homotopy modules over a Koszul algebra. \\

Let $\A$ be a (homogeneous) Koszul algebra with Koszul dual coalgebra $\A^{\ac}$, see \cite[Chapter~$3$]{LodayVallette12}. {We denote its coaugmentation ideal by $\overline{\A}^{\ac}$.} 
Let $(V, d)$ be a chain complex. The set of (left) module structures over the Koszul resolution $\A_\infty:=\Omega \A^{\ac}$ on $V$ is given by the Maurer--Cartan elements 
$$\partial(\alpha)+\alpha \star \alpha=0 $$
of the convolution dg associative algebra 
$$\boxed{\g_{\A, V}:=\big(\Hom(\overline{\A}^{\ac}, \End(V)) , (\partial_V)_*, \star
\big)}\ .$$
One can embed it inside the dg associative algebra $\g_{\A,V}^+$, defined in the same way as before for dg Lie algebras,  but this last one still misses a unit. Instead, we consider the more natural dg unital associative algebra made up of \emph{all the maps} from the Koszul dual coalgebra to the endomorphism algebra 
$$ \g_{\A,V} \mono \g_{\A,V}^+ \mono \boxed{\a_{\A,V}:= \big(
\Hom({\A}^{\ac}, \End(V)),  (\partial_V)_*, \star, {1}
\big)}
\ ,$$
where the unit is given by 
$${1} \ : \ \I \mapsto \id_V $$
and where the last embedding is given by 
$\delta \mapsto  ( \I \mapsto \partial_V)$ .
This associative algebra is graded by the weight grading of the Koszul dual coalgebra:
$$
\a_{\A,V}=\prod_{n\in \NN}
\Hom\big({{\A}^{\ac}}^{(n)}, \End(V)\big) \ .$$ 
Since the Koszul dual coalgebra is \emph{connected}, i.e. 
${{\A}^{\ac}}^{(0)}=\I$ and ${{\A}^{\ac}}^{(n)}=0$ for $n<0$, the convolution unital associative algebra 
$\a_{\A,V}$ decomposes as follows 
$$\Hom\big({\A}^{\ac}, \End(V)\big) \cong 
\End(V)\times\prod_{n=1}^\infty 
\Hom\big({{\A}^{\ac}}^{(n)}, \End(V)\big)  \ . $$ 
So its elements are  series 
$f=f_{(0)}+f_{(1)}+f_{(2)}+\cdots$, where $f_{(n)} : {{\A}^{\ac}}^{(n)} \to \End(V)$ . One can view the gauge group $\Gamma$ in term of group-like elements under the exponential map. 

\begin{proposition}\label{prop:gaugegroupI}
The exponential map 
$$\xymatrix{\Hom(\overline{\A}^{\ac}, \End(V))_0 \ar@<0.5ex>[r]^(0.26)e & \ar@<0.5ex>[l]^(0.74){\ln}
\{f : \A^{\ac} \to \End(V), |f|=0 \ \text{and}\  f(\I)=\id_V \} \subset \Hom({\A}^{\ac}, \End(V))_0} $$
is injective with image the degree $0$ maps such that $f(\I)=\id_V$. Its inverse is provided by the logarithm map 
$$\ln(1+\lambda):=\lambda - \frac{\lambda^{\star 2}}{2} + \frac{\lambda^{\star 3}}{3}+ \cdots \ .$$
\end{proposition}

\begin{proof}
For any $\lambda \in \Hom(\bar{\A}^{\ac}, \End(V))$, its exponential 
$$e^\lambda=1 +\lambda + \frac{\lambda^{\star 2}}{2!} + \frac{\lambda^{\star 3}}{3!} +\cdots  $$
makes sense and is well defined in $\Hom({\A}^{\ac}, \End(V))$ since the image of any element $x\in 
({\A}^{\ac})^{(n)}$ is given by the finite sum 
$$  \lambda(x) + \cdots + \frac{\lambda^{\star n}(x)}{n!}Ê\ . $$
The classical computations hold here, so one has $\ln(e^\lambda)=\lambda$ and $e^{\ln(1+\lambda)}=1+\lambda$ for any $\lambda \in \Hom(\overline{\A}^{\ac},\allowbreak \End(V))$, which concludes the proof. 
\end{proof}

In this case, the exponential group $\mathrm{G}$ exists in the bigger algebra $\a_{\A,V}$, and is isomorphic to the gauge group $\Gamma$. By the weight grading property, its action on Maurer--Cartan elements is well defined and given by the conjugation 
$$\boxed{e^{\ad_\lambda} (\at)=e^{\lambda}\star \at \star e^{-\lambda}}  \ .$$

Since the Maurer--Cartan elements correspond to homotopy $\A$-module structures, the question is now how to give a homotopical interpretation to the gauge group and its action. Let $\alpha$ and $\beta$ be two homotopy $\A$-module structures on $V$, i.e. two Maurer--Cartan elements. One defines a more general notion of maps between two homotopy $\A$-modules, called \emph{$\infty$-morphisms} and denoted $\alpha \rightsquigarrow \beta$, by  degree $0$ maps 
$f : \A^{\ac} \to \End(V)$ satisfying the equation
$${f\star \at = {\beta} \star f}\ . $$ 
These can be composed by the formula $g \star f$ and $\infty$-isomorphisms are the $\infty$-morphisms such that the first component is invertible, i.e. $f_{(0)}(\I)\in \mathrm{GL}(V)$. When this first component is the identity map, we call them \emph{$\infty$-isotopies} and we denote their set by $\infty\mathsf{-iso}$. 

\begin{theorem}\label{thm:DeligneGroupoid}
For any Koszul algebra $\A$ and for any chain complex $(V, d)$, the group of $\infty$-isotopies is isomorphic to the gauge group
$$ \boxed{\Gamma=(g_0, \BCH, 0)\cong (\infty\mathsf{-iso}, \star, \id_V)}  $$ and 
the Deligne groupoid  is isomorphic to the groupoid whose objects are homotopy $\A$-modules and whose 
morphisms are  $\infty$-isotopies
$$\boxed{\mathsf{Deligne}(\g_{A,V})=\left(\mathrm{MC}(\g_{A,V}), \Gamma
\right)=
(\A_\infty\mathsf{-Mod}, \infty\mathsf{-iso})}
\ .$$
\end{theorem}

\begin{proof} 
The first assertion follows directly from  Proposition~\ref{prop:gaugegroupI}.
Two Maurer--Cartan elements $\at$ and ${\beta}$ are gauge equivalent if and only if there is an $\infty$-isotopy between the two homotopy $\A$-module structures on $V$. Indeed, there exists $\lambda \in 
\Hom(\overline{\A}^{\ac}, \End(V))_0$ such that 
${\beta}=e^\lambda.\at=e^{\lambda}\star \at \star e^{-\lambda}$ if and only if 
$e^{\lambda}\star \at= {\beta} \star e^{\lambda}$. 
\end{proof}

\paragraph{\sc Example} The example of the Koszul algebra  $\D=T(\Delta)/(\Delta^2)$ of dual numbers was treated in \cite{DotsenkoShadrinVallette15}. The   gauge group action was used there to  define new invariants on the de Rham cohomology of Poisson manifolds. 

\section{Deformation theory with pre-Lie algebras}\label{sec:DefPreLie}
The notion of pre-Lie algebra sits between the notion of an associative algebra and the notion of a Lie algebra: any associative algebra is an example of a pre-Lie algebra and any pre-Lie product induces a Lie bracket. In this section, we first develop the integration theory of pre-Lie algebras, that is how to get a group via exponentials from a pre-Lie algebra. Then, we apply it to make explicit the deformation theory controlled by pre-Lie algebras, that is we describe the gauge equivalence via the action of this latter group.    \\

Let us now study the more general  situation where the 
dg Lie algebra $\g=(g, d, [\, , ])$ is actually coming from a dg left unital  pre-Lie algebra $\mathfrak{a}=(A, d, \star, 1)$:  
under the usual anti-symmetrization: $[x,y]:=x\star y - (-1)^{|x||y|} y\star x$. 
Recall, for instance from \cite[Section~$13.4$]{LodayVallette12}, that a \emph{pre-Lie algebra} is defined by a binary product whose associator is right symmetric: 
$$(x\star y)\star z - x\star (y\star z)=(-1)^{|y||z|}\big( (x\star z)\star y - x\star (z\star y)\big) \ . $$ 
Here the unit is only supposed to hold on the left-hand side: $1 \star x= x$. 

We will again assume the existence of similar underlying weight grading  satisfying 
\begin{eqnarray}\label{eqn:weight} 
\quad\quad
A\cong\prod_{n=0}^\infty A^{(n)}, \ 1\in A^{(0)}, \ A^{(n)}\star A^{(m)} \subset A^{(n+m)}, \quad \text{and} 
 \quad g\cong \prod_{n=1}^\infty A^{(n)}  \quad \text{so} 
\quad A\cong A^{(0)}\times {g}\ . 
\end{eqnarray}
Under this assumption, one can realize  the gauge group $\Gamma$ as exponentials  as follows. For pre-Lie algebras, one chooses to iterate the product on the right-hand side: 
$$\boxed{\lambda^{\star n}:=\underbrace{(\cdots((\lambda \star \lambda) \star \lambda)\cdots )\star \lambda}_{n\  \text{times}}}\ . $$
The \emph{pre-Lie exponential} of elements $\lambda$ of $\a$ with trivial weight $0$ component is then 
defined as a series as follows: 
$$e^\lambda:=1 +\lambda + \frac{\lambda^{\star 2}}{2!} + \frac{\lambda^{\star 3}}{3!} +\cdots \ . $$
(Notice that the space of vector fields of a manifolds with a flat and torsion free connection admits a pre-Lie product such that the solutions to the flow differential equation are given by such  series, see \cite{Vinberg63, AgrachevGamkrelidze80} for more details.) 

We call \emph{group-like element} any degree $0$ element of $\a$ whose weight $0$ component is equal to $1$. We denote the associated set by $\mathrm{G}$. 

\begin{lemma}\label{lemma:ExponentialII}
The exponential map 
$$e \ : \ g_0 \xrightarrow{\cong}
\mathrm{G}=\{ 1 \}\times g_0 = \{ 1 \}\times \prod_{n \ge 1 } A^{(n)}_0\subset A$$
is injective with image the set of group-like elements. 
\end{lemma}

\begin{proof}
Any element $\lambda \in g_0$ viewed in $\a$ decomposes according to the weight grading as follows 
$$\lam = \lam_{(1)} + \cdots + \lam_{(n)}+\cdots\ , $$
where $\lam_{(n)}\in A^{(n)}$. So the exponential of $\lam$ is equal to 
\begin{eqnarray*}
e^\lam & =& 1 + \lam + {\textstyle \frac12} \lam\star \lam + {\textstyle \frac16} (\lam \star \lam)\star \lam+ \cdots \\
&=& 1 + \underbrace{\lam_{(1)}}_{(1)}+\underbrace{\lam_{(2)} + {\textstyle \frac12} \lam_{(1)}\star \lam_{(1)}}_{(2)}  \\&& 
+\underbrace{\lam_{(3)} + {\textstyle \frac12} \lam_{(1)}\star \lam_{(2)} +  {\textstyle \frac12} \lam_{(2)}\star \lam_{(1)} +{\textstyle \frac16} (\lam_{(1)}\star \lam_{(1)})\star \lam_{(1)} 
}_{(3)} + \cdots
\end{eqnarray*}
From this decomposition, it is easy to see that the equation $e^\lam=1+a$, for any $a\in  g_0=\prod_{n \ge 1 } A^{(n)}_0$ admits one and only one solution $\lam\in g_0$ given by: 
\begin{eqnarray*}
\lam_{(1)}&=&a_{(1)}\\
\lam_{(2)}&=&a_{(2)}-{\textstyle \frac12} a_{(1)}\star a_{(1)}\\
\lam_{(3)}&=&a_{(3)}-{\textstyle \frac12} a_{(1)}\star a_{(2)}-{\textstyle \frac12} a_{(2)}\star a_{(1)}
+{\textstyle \frac14} (a_{(1)}\star a_{(1)})\star a_{(1)}
+{\textstyle \frac{1}{12}} a_{(1)}\star (a_{(1)}\star a_{(1)})\\
\vdots &=& \vdots 
\quad .
\end{eqnarray*}
\end{proof}
Lemma~\ref{lemma:ExponentialII} shows that the pre-Lie exponential map admits an inverse,  which plays the role of a pre-Lie logarithm. It is actually given by a series $\Omega(a)$ called the ``pre-Lie Magnus expansion'' series, which begins by 
$$\ln(1+a)=\Omega(a):=a - \frac12 a\star a + \frac14 a\star (a \star a)
+ \frac{1}{12}(a\star a)\star a+\cdots \ , $$
 see \cite{AgrachevGamkrelidze80, Manchon11} for more details. \\

Recall that, in any pre-Lie algebra, one defines  \emph{symmetric braces} operations by the following formulae:
$$\begin{array}{lcl}
\{ a; \}&:=&a \\
\{ a; b_1\}&:=&a\star b_1 \\
\{ a; b_1,b_2\}&:=& \{\{a; b_1\};  b_2\} - \{a; \{ b_1;  b_2\}\}= (a\star b_1)\star b_2 - a\star (b_1 \star b_2)\\
&&\\
\{ a; b_1,\ldots, b_n\}&:=&  \{ \{a; b_1,\ldots, b_{n-1}\}; b_n\}- \displaystyle \sum_{i=1}^{n-1} \{a; b_1, \ldots, b_{i-1}, 
\{b_i; b_n\}, b_{i+1}, \ldots, b_{n-1}\}   \ .\\
\end{array}
$$
One can see that these operations are {(graded)} symmetric with respect to the $n$ right-hand side inputs. They satisfy some relations, see \cite[Section~$13.11.4$]{LodayVallette12}. The category of pre-Lie algebras is isomorphic to the category of symmetric braces algebras \cite{GuinOudom08}. 

In any pre-Lie algebra like $\a$, we define the \emph{circle product} as the sum of the symmetric braces 
$$\boxed{a \circledcirc (1+b) := \sum_{n\ge 0}  {\displaystyle \frac{1}{n!}}  \{a; \underbrace{b, \ldots, b}_{n}\}}\ ,$$
for any $a, b\in A$ such that $b$ has trivial weight zero component. One can see that the circle product is linear only on the left-hand side, i.e. with respect to $a$, and that $1$ is a unit for it: $a\circledcirc 1 =a$ and $1 \circledcirc (1+b)=1+b$. 

\begin{proposition}\label{prop:MAINrelation}
Any pre-Lie algebra $\a$, with the weight grading condition (\ref{eqn:weight}), satisfies:
$$\boxed{e^{r_\lambda}(a)=a\circledcirc e^\lambda}\ ,   $$ 
for any $a, \lambda \in A$, with $\lambda$ having trivial weight zero component and where the operator $r_\lambda(a):=a \star \lambda$ presents the right product by $\lambda$. 
\end{proposition}

\begin{proof}
It is enough to  prove this result in the free pre-Lie algebra on two generators $\lambda$ and $a$. 
Recall from \cite{ChapotonLivernet01} that this free pre-Lie algebra is given by rooted trees with vertices labelled by $\lambda$ and $a$, where the pre-Lie product of two such trees is equal to the sum over all possible ways to graft the root of the  second tree at the top of a vertex of the first one. The unit is given by the trivial tree 
$\vcenter{\xymatrix@R=8pt{ *{} \ar@{-}[d]\\ 
*{}}}$
 with no vertices. 

For any rooted tree $t\in \mathsf{RT}$, we denote by $\nu_t$ the number of its vertices and by $n_t$ by number of ways to put its vertices on levels with one and only one vertex per level. (These numbers are also called the \textit{Connes--Moscovici weights}, see \cite{Kreimer99, Brouder00}). We will call this process ``levelization''.  In the example of the tree 
$$ 
t=
\vcenter
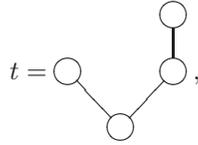
{
\xymatrix@M=5pt@R=10pt@C=10pt{
 & & *+[o][F-]{}\ar@{-}[d]\\
*+[o][F-]{}\ar@{-}[dr] &&  *+[o][F-]{}\ar@{-}[dl] \\
& *+[o][F-]{}
}}\ ,
$$
the number of levelizations is equal to $n_t=3$ since they are only 
$$ 
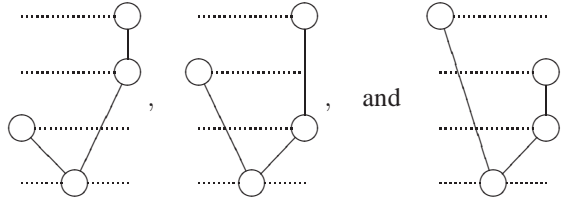
\vcenter{
\xymatrix@M=5pt@R=10pt@C=10pt{
*{}\ar@{..}[rr] & & *+[o][F-]{}\ar@{-}[d]\\
*{}\ar@{..}[rr] &&  *+[o][F-]{}\ar@{-}[ddl] \\
*+[o][F-]{}\ar@{-}[dr]\ar@{..}[rr] &&  *{}\\
*{}\ar@{..}[r]& *+[o][F-]{} & *{}\ar@{..}[l]
}}\ , \quad 
\vcenter{
\xymatrix@M=5pt@R=10pt@C=10pt{
*{}\ar@{..}[rr] & & *+[o][F-]{}\ar@{-}[dd]\\
*+[o][F-]{}\ar@{-}[ddr]\ar@{..}[rr] &&  *{} \\
*{}\ar@{..}[rr] &&  *+[o][F-]{}\ar@{-}[dl]\\
*{}\ar@{..}[r]& *+[o][F-]{} & *{}\ar@{..}[l]
}}\ ,
\quad \text{and} \quad
\vcenter{
\xymatrix@M=5pt@R=10pt@C=10pt{
*+[o][F-]{}\ar@{-}[dddr]\ar@{..}[rr] & & *{}\\
*{}\ar@{..}[rr] &&   *+[o][F-]{}\ar@{-}[d]  \\
*{}\ar@{..}[rr] &&  *+[o][F-]{}\ar@{-}[dl]\\
*{}\ar@{..}[r]& *+[o][F-]{} & *{}\ar@{..}[l]
}}\ .
$$
Since the formula for the pre-Lie exponential 
$e^\lambda:=1 +\lambda + \frac{\lambda^{\star 2}}{2!} + \frac{\lambda^{\star 3}}{3!} +\cdots$
is defined by iterating the product $\star$ on the right-hand side, it amounts to successively graft the tree with one vertex labelled by  $\lambda$ on the top of the root vertex labelled by $\lambda$. Therefore, the exponential of $\lambda$ is equal to the following series of rooted trees 
$$e^\lambda=\sum_{t\in \mathsf{RT}} \frac{n_t}{\nu_t!} \, t(\lambda)\ , $$
where the notation $t(\lambda)$ stands for the rooted trees $t$ with all vertices labelled by $\lambda$. 

Since the element 
$$e^{r_\lambda}(a)=a +  a \star \lambda + {\textstyle \frac12} (a\star \lambda)\star \lambda + \cdots$$
is produced by the same kind of formula, it is equal to 
$$e^{r_\lambda}(a)=\sum_{t \in  \mathsf{rRT}} \frac{n_t}{(\nu_t-1)!} \, t(a,\lambda)\ , $$
where the sum runs  over reduced rooted trees $ \mathsf{rRT}$, i.e. trees with at least one vertex, and where the notation $t(a,\lambda)$ stands for the rooted trees $t$ with the root vertex labelled by $a$ and all other vertices labelled by $\lambda$.

One can easily prove by induction on $n$ that the $n$th symmetric brace in the free pre-Lie algebra is equal to the labelled corolla with $n+1$ vertices
$$\{a, \underbrace{b, \ldots, b}_{n}\}=
\vcenter{\xymatrix@M=5pt@R=10pt@C=10pt{
*+[o][F-]{b}\ar@{-}[dr]^1 & \cdots & *+[o][F-]{b}\ar@{-}[dl]_n \\
& *+[o][F-]{a} \ar@{-}[u] &}}
\ .
$$
Therefore, the circle product $a\circledcirc e^\lambda$ is equal to the sum
$$a\circledcirc e^\lambda=\sum_{k\ge 0}\,  \sum_{t_1, \ldots, t_k  \in \mathsf{rRT}} 
\frac{n_{t_1}\cdots n_{t_k}}{\nu_{t_1}!\cdots \nu_{t_k}!} \frac{n(t_1, \ldots, t_k)}{k!}\ 
\vcenter{\xymatrix@M=5pt@R=10pt@C=10pt{
*+[F-]{t_1}\ar@{-}[dr] & *+[F-]{t_2}\ar@{-}[d] & \cdots & *+[F-]{t_k}\ar@{-}[dll] \\
& *+[o][F-]{a} \ar@{-}[u] && }}
\ , 
$$
where $n(t_1, \ldots, t_k)$ is equal to the number of times the configuration of trees $t_1, \ldots, t_k$ appears in the symmetric $k$-fold product $e^\lambda\cdots e^\lambda$. Notice that, if we denote by $i_1,\ldots, i_l$ respectively the number of trees of type $1,\ldots, l$ in $( t_1, \ldots, t_k)$, then 
$$\frac{n(t_1, \ldots, t_k)}{k!}=\frac{1}{i_1!\cdots i_l!}\ . $$

To conclude, let us consider a rooted tree $t(a, \lambda)$ with the root vertex labelled by $a$ and all other vertices labelled by $\lambda$. We first denote by $t_1, \ldots, t_k$ the subtrees which are attached to the root vertex $a$. They individually admit $n_{t_1}, \ldots, n_{t_k}$  levelizations. Given one levelization for each of these trees, the number of ways to induce a  levelization for the global tree $t$ is equal to the number of $(\nu_{t_1}, \ldots, \nu_{t_k})$-shuffles, that is 
$$\frac{(\nu_{t_1}+\cdots+\nu_{t_k})!}{\nu_{t_1}!\cdots \nu_{t_k}!}=
\frac{(\nu_{t}-1)!}{\nu_{t_1}!\cdots \nu_{t_k}!}\ .$$
Then, we forget about the labeling $t_1, \ldots, t_k$ of the subtrees. This gives 
$\frac{k!}{n(t_1, \ldots, t_k)}$ 
times the same levelization of the tree $t$. 
In the end, the number of levelizations of a tree $t$ is equal to 
$$n_t=\frac{(\nu_{t}-1)!}{\nu_{t_1}!\cdots \nu_{t_k}!} \frac{n(t_1, \ldots, t_k)}{k!}\ , $$
which concludes the proof. 
\end{proof}

\begin{remark}
 A  similar kind of formula, relating the two products $\star$ and $\cc$, can be found in \cite[Formula~$(58)$]{Manchon15}. 
\end{remark}

\begin{theorem}\label{thm:MainIso}
Under the pre-Lie exponential map, the gauge group is isomorphic to the group made up of group-like elements equipped with the circle product:
$$\boxed{\Gamma=\left( g_0, \BCH, 0
\right)\cong  
\left(\mathrm{G}, \circledcirc, 1\right)=\mathfrak{G}}
\ .
$$
\end{theorem}

\begin{proof}
It is known that 
$$e^{\BCH(\mu, \lambda)}=e^{r_\lambda}(e^\mu) \ , $$ 
see \cite{Manchon11}. From Proposition~\ref{prop:MAINrelation}, we get 
$$e^{\BCH(\mu, \lambda)}=e^\mu\circledcirc e^\lambda \ , $$ 
which concludes the proof, with Lemma~\ref{lemma:ExponentialII}.
\end{proof}
Notice that the group $\mathfrak{G}$ is equal to the ``group of formal flows'' of \cite{AgrachevGamkrelidze80}, which was defined without the circle product. 

\begin{corollary}\label{cor:AssocUnit}
The circle product $\circledcirc$ is associative and, on the space $\lbrace1\rbrace\times g$,  the inverse of $e^\lambda$ for $\circledcirc$ is $e^{-\lambda}$.
\end{corollary}

\begin{proof}
This is obtained by the transport of structure under the isomorphism of Theorem~\ref{thm:MainIso}. 
\end{proof}

When one does not know the exponential form of a group-like element $1-\mu$, with $\mu\in g$, its inverse, for the circle product, can be described as follows. We consider the free pre-Lie algebra  $\mathrm{preLie}(a)$
on one generator $a$, which is given by rooted trees (with vertices labelled by $a$). By the universal property, there exists one morphism of pre-Lie algebras from $\mathrm{preLie}(a)$ to $\a$. For any rooted tree $t$, we denote  by $t(\mu)$ the image 
of the tree $t$ 
under this morphism.

A morphism between two trees is a map sending  vertices to vertices and edges to edges which respects the adjacency condition. We denote by $\mathrm{Aut}\, t$ the group of automorphisms of a tree $t$ and by  $|\mathrm{Aut}\, t|$ its  cardinal.

\begin{proposition}\label{prop:ccInverse}
In any preLie algebra $\a$ satisfying the weight grading assumption (\ref{eqn:weight}), the inverse, for the circle product, of a group-like element $1-\mu$, with $\mu\in g$, is given by 
$$(1-\mu)^{\cc -1}=  \sum_{t\in \mathsf{RT}} 
\frac{1}{|\mathrm{Aut}\, t|}\, t(\mu)\  .$$
\end{proposition}

\begin{proof}
It is enough to prove that 
$$(1-a)\cc  \left(\sum_{t\in \mathsf{RT}} 
\frac{1}{|\mathrm{Aut}\, t|}\, t(a)\right)=1\  .$$
in the free pre-Lie algebra $\mathrm{preLie}(a)$. The left-hand term is equal to a sum over the set of rooted trees. The trivial tree $\vcenter{\xymatrix@R=8pt{ *{} \ar@{-}[d]\\ 
*{}}}$
  appears only once. Let $s\in \mathsf{rRT}$ denote a non-trivial rooted tree, that is with at least one vertex. There are two ways to get the term $s(a)$. The first way comes with the element $1$ of $(1-a)$; so the coefficient is $\frac{1}{|\mathrm{Aut}\, s|}$. The other way comes with the element $a$ of $(1-a)$. The non-trivial tree $s$ is equal to a root  vertex grafted by trees $t_1, \ldots, t_k$. 
  Therefore, the coefficient of the element $s(a)$ arising in that way is equal to 
$$-\frac{ |\Sy_{(t_1, \ldots, t_k)}|}{k!} \cdot \frac{1}{|\mathrm{Aut}\, t_1|} \cdots  \frac{1}{|\mathrm{Aut}\, t_k|} 
\ ,$$
where $\Sy_{(t_1, \ldots, t_k)}$ denotes the permutation group of the $k$-tuple $(t_1, \ldots, t_k)$.
Notice that, if we denote by $i_1,\ldots, i_l$ respectively the number of trees of type $1,\ldots, l$ in $( t_1, \ldots, t_k)$, then 
$$\frac{ |\Sy_{(t_1, \ldots, t_k)}|}{k!}=\frac{i_1!\cdots i_l!}{k!}=\frac{1}{n(t_1, \ldots, t_k)}\ . $$
However,  the automorphism group  of the tree $s$ is 
  isomorphic to 
  $$\mathrm{Aut}\, s\ \cong\  \mathrm{Aut}(t_1, \ldots, t_k) \times \mathrm{Aut}\, t_1\times \cdots \times \mathrm{Aut}\, t_k \ , $$
  where the first term denotes the automorphism group of the corolla where the top vertices are labelled respectively by $t_1, \ldots, t_k$. Therefore, the coefficient of the element $s(a)$ arising in that way is equal to $-\frac{1}{|\mathrm{Aut}\, s|}$, which concludes the proof. 
\end{proof}

So, in this case, the exponential group $\mathrm{G}$ exists in the bigger algebra $\a$ and realizes the gauge group $\Gamma$ coming from the Lie algebra $\g$. Using this realization of the gauge group, the action on the variety of Maurer--Cartan elements is given by the following formula. 

\begin{proposition}\label{prop:GaugeActionPreLie}
Under the abovementioned weight grading assumption, the infinitesimal  Lie algebra $\g_0$ action  defined in Section~\ref{subsec:InfAction} integrates to $t=1$ and is equal to the following action of the gauge group  on the variety of Maurer--Cartan elements : 
$$ \boxed{\lambda.\at=e^{\ad_\lambda}(\at)= \left(e^\lambda \star \at\right) \circledcirc e^{-\lambda}}   \ .$$
\end{proposition}

\begin{proof}
The weight grading assumption ensures the convergence of $e^{t\ad_\lambda}(\at)$ at $t=1$ to $e^{\ad_\lambda}(\at)$. 
The function $\gamma(t):=e^{t\ad_\lambda}(\at)$ is the unique solution to the differential equation 
$${\gamma}'(t)=\ad_\lambda({\gamma}(t))=\lambda\star \gamma(t) - \gamma(t)\star \lambda$$ such that $\gamma(0)=\at$.
Let us show that the function 
$\varphi(t):=\left(e^{t\lambda} \star \at\right) \circledcirc e^{-t\lambda}$
satisfies the same differential equation, which will conclude the proof. First, the value of the function $\varphi$ at $t=0$ is 
$$\varphi(0)=(1\star \at)\circledcirc 1=\at\ .$$
Then, we have 
$$\varphi'(t):=\underbrace{\left((e^{t\lambda}\star \lambda) \star \at\right) \circledcirc e^{-t\lambda}}_{(i)}
-\underbrace{\left(e^{t\lambda} \star \at\right) \circledcirc (e^{-t\lambda}; e^{-t\lambda}\star \lambda)}_{(ii)}\ , $$
where the notation of the right-hand term stands for 
$$a\circledcirc (1+b; c):=\sum_{n\ge 0} {\displaystyle\frac{1}{n!}}\{a; \underbrace{b, \ldots, b}_{n}, c\} \ .$$
Let us first prove the general formula 
\begin{eqnarray}\label{eqn:Prelie1}a\circledcirc (1+b; (1+b)\star c)=(a\circledcirc (1+b))\star c\ . 
\end{eqnarray}
As in the proof of Proposition~\ref{prop:MAINrelation}, let us perform this computation is the free pre-Lie algebra on three generators $a$, $b$, and $c$. Notice that, by linearity, the left-hand side is equal to 
\begin{eqnarray*}
a\circledcirc (1+b; (1+b)\star c)&=&a\circledcirc (1+b; c)+a\circledcirc (1+b; b\star c)   \\
&=&\sum_{n\ge 0} \frac{1}{n!} \ 
\vcenter{\xymatrix@M=5pt@R=10pt@C=10pt{
*+[o][F-]{b}\ar@{-}[dr]^1 & \cdots & *+[o][F-]{b}\ar@{-}[dl]_n &*+[o][F-]{c}\ar@{-}[dll] \\
& *+[o][F-]{a} \ar@{-}[u]&&
}}+\sum_{n\ge 0} \frac{1}{n!} \ 
\vcenter{
\xymatrix@M=5pt@R=10pt@C=10pt{
 & & &*+[o][F-]{c}\ar@{-}[d]\\
*+[o][F-]{b}\ar@{-}[dr]^1 & \cdots & *+[o][F-]{b}\ar@{-}[dl]_n &*+[o][F-]{b}\ar@{-}[dll] \\
& *+[o][F-]{a}\ar@{-}[u]&&
}}\ .
\end{eqnarray*}
The right-hand term is equal to 
\begin{eqnarray*}
(a\circledcirc (1+b))\star c= 
\left(\sum_{n\ge 0} \frac{1}{n!} \ 
\vcenter{\xymatrix@M=5pt@R=10pt@C=10pt{
*+[o][F-]{b}\ar@{-}[dr]^1 & \cdots & *+[o][F-]{b}\ar@{-}[dl]_n \\
& *+[o][F-]{a} \ar@{-}[u]&}}
\right)\star c\ .
\end{eqnarray*}
Since the pre-Lie product $\star$ amounts to grafting a vertex $c$ above the rooted trees labeled by $a$ and $b$, then there are two possibilities: one can graft the vertex $c$ above the root vertex $a$ or above one of the $n$ vertices $b$. So we recover precisely the two above sums, which proves Formula~(\ref{eqn:Prelie1}). 

Formula~(\ref{eqn:Prelie1}) applied to $a=e^{t\lambda} \star \at$, $1+b=e^{-t\lambda}$, and $c=\lambda$ proves 
that 
$$(ii)= \left(\left(e^{t\lambda} \star \at\right)\circledcirc e^{-t\lambda}\right) \star \lambda=\varphi(t)\star \lambda\ .$$

Applying Proposition~\ref{prop:MAINrelation} to $a=\lambda$ gives $e^{t\lambda}\star \lambda = \lambda \circledcirc e^{t \lambda}$ and then $(i)= \left( \left(\lam \circledcirc e^{t \lam}\right)\star \at   \right)\circledcirc e^{-t\lam}$.
Let us now prove the general formula 
\begin{eqnarray}\label{eqn:Prelie2}
 \left( \left(a \circledcirc e^{ \lam}\right)\star b   \right)\circledcirc e^{-\lam}
 =
 a\star \left( \left(e^\lam \star b\right)\cc e^{-\lam}\right)
 \ . 
\end{eqnarray}
Formula~(\ref{eqn:Prelie1}) gives  
$$ \left( \left(a \circledcirc e^{ \lam}\right)\star b   \right)\circledcirc e^{-\lam} =
 \left( a \circledcirc \left(e^{ \lam}; e^{\lam}\star b   \right)\right)\circledcirc e^{-\lam}
 = \left( (1+a) \circledcirc \left(e^{ \lam}; e^{\lam}\star b   \right)\right)\circledcirc e^{-\lam}
 -\left( e^\lam \star b\right) \cc e^{-\lam}\ .
$$
The associativity formula of Corollary~\ref{cor:AssocUnit}
$$\left((1+a)\cc \left(e^\lam +t e^\lam\star b \right)\right) \cc e^{-\lam}
=
(1+a)\cc \left(\left(e^\lam +t e^\lam\star b \right) \cc e^{-\lam}\right) \ .$$ 
for the circle product when derived and applied at $t=0$ gives
$$\left( (1+a) \circledcirc \left(e^{ \lam}; e^{\lam}\star b   \right)\right)\circledcirc e^{-\lam}
=
 (1+a) \circledcirc \left(e^{ \lam}\cc e^{-\lam}; \left(e^{\lam}\star b   \right) \cc  e^{-\lam}\right)\ .
 $$
Since $e^{-\lam}$ is the inverse of $e^\lam$ for the circle product, by  Corollary~\ref{cor:AssocUnit}, we finally get 
\begin{eqnarray*}
\left( \left(a \circledcirc e^{ \lam}\right)\star b   \right)\circledcirc e^{-\lam}&=& 
(1+a)\star \left( \left( e^\lam \star b \right) \cc e^{-\lam}\right)
 -\left( e^\lam \star b\right) \cc e^{-\lam}\\
 &=& 
a\star \left( \left( e^\lam \star b \right) \cc e^{-\lam}\right)\ .
\end{eqnarray*}
Formula~(\ref{eqn:Prelie2}) applied to $a=\lam$, $b=\at$, and $\lambda=t\lambda$ proves 
that 
$$(i)= \lam \star \left( \left(e^{t\lam} \star \at\right)\cc e^{-t\lam}\right)=\lambda  \star \varphi(t)
\ .$$
Therefore the function $\varphi(t)$ satisfies the same differential equation as the function $\gamma(t)$, that is 
$$\varphi'(t)=\lambda  \star \varphi(t) - \varphi(t)\star \lambda\ . $$ Hence, $\varphi(1)=\gamma(1)$, which is the formula of the proposition. 
\end{proof}

\paragraph{\sc Remark} Notice that when the pre-Lie product $\star$ is actually associative, we recover all the formulae and results of the previous sections. 

\section{Example : homotopy algebras over a Koszul operad}\label{sec:HomoAlgOp}
We can now apply the general pre-Lie deformation theory of the previous section to the case of homotopy algebras over a Koszul operad. This allows us to describe the associated Deligne groupoid in homotopical terms. This uses the crucial fact that the deformation theory of algebras over an operad is controlled by a Lie algebra coming from a pre-Lie algebra. 
\\

Let $\P$ be a (homogeneous) Koszul operad with Koszul dual cooperad $\P^{\ac}$. Let $(V, d)$ be a chain complex. Recall from \cite[Chapter~10]{LodayVallette12} that the set of algebra structures over the Koszul resolution $\P_\infty=\Omega \P^{\ac}$ on $V$ is given by the Maurer--Cartan elements 
$$\partial(\alpha)+\alpha \star \alpha=0 $$
of the convolution dg pre-Lie algebra 
$$\boxed{\g_{\P, V}:=\big(\Hom_\Sy(\overline{\P}^{\ac}, \End_V) , (\partial_V)_*, \star
\big)}\ ,$$
with the pre-Lie product is defined by 
$$f\star g : \oPa
 \xrightarrow{\Delta_{(1)}} \oPa \otimes \oPa
  \xrightarrow{f\otimes g}  \End_V \otimes \End_V 
  \xrightarrow{\gamma}  \End_V \ ,$$
{  where $\Delta_{(1)}$ is the infinitesimal decomposition map of the cooperad $\P^{\ac}$, which splits operations into two, see \cite[Section~$6.1.4$]{LodayVallette12}, and where $\gamma$ is the composition map of the endomorphism operad $\End_V$.}
Again, we choose to embed it inside the dg left unital pre-Lie algebra made up of \emph{all the maps} from the Koszul dual cooperad to the endomorphism operad 
$$ \g_{\P,V} \mono  \boxed{\a_{\P,V}:= \big(
\Hom_\Sy({\P}^{\ac}, \End_V),  (\partial_V)_*, \star, {1}
\big)}
\ ,$$
where the left unit is again given by 
${1} \ : \ \I \mapsto \id_V $ and where the second embedding is given by $\delta \mapsto (I\mapsto \partial_V)$.
(Notice that the trick which consists in considering the one-dimensional extension of $\g_{\P,V}$ to make the differential internal does not work for dg pre-Lie algebras.)

This preLie algebra is graded by the weight of the Koszul dual cooperad
$$\Hom_\Sy\big({{\P}^{\ac}}^{(n)}, \End_V\big)\ , $$ that is 
$$\Hom_\Sy\big({\P}^{\ac}, \End_V\big) =
\End(V)\times\prod_{n=1}^\infty 
\Hom_\Sy\big({{\P}^{\ac}}^{(n)}, \End_V \big)  \ . $$ 
Its elements are series 
$f=f_{(0)}+f_{(1)}+f_{(2)}+\cdots$, where $f_{(n)} :{{\P}^{\ac}}^{(n)}\to \End_V$.\\

In this pre-Lie algebra, the circle product admits the following closed description. 
\begin{lemma}\label{lem:circledcirc}
In the pre-Lie algebra $\a_{\P,V}$ the circle product of two maps 
$f,g$ such that $f(\I)=g(\I)=\id_V$ 
 is equal to 
$$f\cc g : \P^{\ac}
 \xrightarrow{\Delta} \P^{\ac}\circ \P^{\ac} 
  \xrightarrow{f\circ g}  \End_V \circ \End_V 
  \xrightarrow{\gamma}  \End_V \ .$$
\end{lemma}

\begin{proof}
We use the splitting $\P^{\ac}=\I\oplus \overline{\P}^{\ac}$ to write the components of the decomposition map $\Delta$ on $\overline{\P}^{\ac}$ as follows
$$\oPa \xrightarrow{\Delta} 
\oPa \circ \I \ \oplus \  \I \otimes \oPa \ \oplus\  
\bigoplus_{k\ge 1} 
\oPa \otimes (\underbrace{\oPa \otimes \cdots \otimes \oPa}_{k})  \ ,$$
where, in the notation $\oPa \otimes ({\oPa \otimes \cdots \otimes \oPa}) $, we have suppressed the superfluous identity elements $\I$ (and some symmetric group elements). 
Let us denote by 
$$\Delta_{(0)} : \oPa \xrightarrow{\Delta}  \oPa \circ \I \ \oplus \  \I \otimes \oPa  \quad \text{and by} \quad 
\Delta_{(k)} : \oPa \to  \oPa \otimes (\underbrace{\oPa \otimes \cdots \otimes \oPa}_{k}) $$
the components of the decomposition map $\Delta$. If we denote $f=1+\bar f$ and $g=1+\bar g$, then we get 
$$f\cc g = 1 +\bar f +\bar g + \bar f\star\bar g + \textstyle{ \frac12 }\left\{ \bar f; \bar g, \bar g\right\}
+ \textstyle{ \frac16 }\left\{ \bar f; \bar g, \bar g, \bar g\right\}+ \cdots\ .$$ 
First,  both formulae agree on $\I$ and give $1$ :
$$ \I  \xrightarrow{\Delta} \I \circ \I
  \xrightarrow{f\circ g}  \id_V\otimes\id_V
  \xrightarrow{\gamma}  \id_V \ .$$
Then, on $\oPa$, the component $\Delta_{(0)}$ of $\Delta$ gives $\bar f + \bar g$.
It remains to prove, by induction on $k\ge 1$,  that the composite 
$$\oPa
 \xrightarrow{\Delta_{(k)}} \oPa \otimes (\underbrace{\oPa \otimes \cdots \otimes \oPa}_{k}) 
  \xrightarrow{\bar f\otimes (\bar g\otimes \cdots \otimes \bar g)}  \End_V \otimes ( \End_V \otimes 
  \cdots \otimes \End_V ) 
  \xrightarrow{\gamma}  \End_V \ .$$
is equal to 
$$  { \frac{1}{k!} }\big\{ \bar f; \underbrace{\bar g, \ldots, \bar g}_{k}\big\}\ .$$
The map $\Delta_{(1)}$ gives $\bar f \star \bar g$ by definition. Suppose now the relation true up to $k$. The $k+1$ brace operations is defined by 
$$\frac{1}{(k+1)!}\big\{ \bar f; \underbrace{\bar g, \ldots, \bar g, \bar g}_{k+1}\big\}= 
\frac{1}{k+1}\left(
\big\{ \frac{1}{k!} \{\bar f; \underbrace{\bar g, \ldots, \bar g}_{k}\}; \bar g     \big\}-
\frac{1}{(k-1)!}\big\{  \{\bar f; \underbrace{\bar g, \ldots, \bar g}_{k-1}, \{\bar g ; \bar g\} \big\}
\right)\ .$$
By the induction assumption, the first term is obtained by applying $\Delta_{(1)}  : \oPa \to \oPa \otimes \oPa$ first and then $\Delta_{(k)}$ to the left-hand side $\oPa$. This produces the  two kinds of elements depending on whether the right-hand side $\oPa$ is directly attached to the root $\oPa$ or whether there is an element of $\oPa$ in between. The second term of the formula cancels precisely these latter elements. By the coassociativity of the decomposition map $\Delta$, the former elements are given by $\frac{1}{k+1}$ times the elements produced by the map 
$$ \oPa
\to \oPa \otimes (\underbrace{\oPa \otimes \cdots \otimes \oPa}_{k}\otimes \oPa) \ , $$
with one element emphasized on the right-hand side; it is therefore equal to the terms produced by $\Delta_{(k+1)}$, which concludes the proof. 
\end{proof}

\begin{notation}
We will still denote by $f\cc g$ the operation in the convolution pre-Lie algebra $\a_{\P,V}$ defined by the composite $\gamma(f \circ g)\Delta$, for any $f,g$, even when their first component is not the identity. This generalization of the product $\cc$ satisfies the same properties, like the associativity relation for instance. 
\end{notation}

Since the Maurer--Cartan elements of $\a_{\P, V}$ correspond to homotopy $\P$-algebra structures, the question is now how to give a homotopical interpretation of the gauge group and its action. Let $\alpha$ and $\beta$ be two homotopy $\P$-algebra structures on $V$, i.e. two Maurer--Cartan elements. One defines a more general notion of maps between two homotopy $\P$-algebras, called \emph{$\infty$-morphisms} and denoted $\alpha \rightsquigarrow \beta$, by  degree $0$ elements 
$f : \P^{\ac} \to \End_V$ satisfying the equation
$$f \star \at = \beta \cc f\ . $$ 
(Notice the discrepancy  of this formula: the $\star$-product lies on the left-hand side and the $\cc$-product lies on the right-hand side.) 
They can be composed using the formula $f \cc g$ and $\infty$-isomorphisms are the $\infty$-morphisms such that the first composite is invertible, i.e. $f_{(0)}(\I)\in \mathrm{GL}(V)$. When this first component is the identity map, we call them \emph{$\infty$-isotopies}. 
We refer the reader to \cite[Section~10.2]{LodayVallette12} for more details.  

\begin{theorem}\label{thm:DeligneGroupoidII}
For any Koszul operad $\P$ and for any chain complex $(V, d)$, the group of $\infty$-isotopies is isomorphic to the gauge group under the preLie exponential map
$$ \boxed{\Gamma=(g_0, \BCH, 0)\cong (\infty\textsf{-}\mathsf{iso}, \cc, \id_V)}  \ . $$ 
The Deligne groupoid  is isomorphic to the groupoid whose objects are homotopy $\P$-algebras and whose 
morphisms are  $\infty$-isotopies
$$\boxed{\mathsf{Deligne}(\g_{\P,V})=\left(\mathrm{MC}(\g_{\P,V}), \Gamma
\right)\cong
(\P_\infty\textsf{-}\mathsf{Alg}, \infty\textsf{-}\mathsf{iso})}
\ .$$
\end{theorem}

\begin{proof} 
The first assertion follows directly Theorem~\ref{thm:MainIso} and Lemma~\ref{lem:circledcirc}.
The second assertion is given by   Proposition~\ref{prop:GaugeActionPreLie} as follows.
Two Maurer--Cartan elements $\at$ and ${\beta}$ are gauge equivalent if and only if there is 
an element $\lambda \in 
\Hom(\overline{\P}^{\ac}, \End(V))_0$ such that 
$${\beta}=e^{\mathrm{ad}_\lambda}(\at)=(e^{\lambda}\star \at) \cc e^{-\lambda}\ .$$
By considering the circle product by $e^\lam$ on the right, it is equivalent to 
$$e^\lam\star \at = {\beta} \circledcirc e^\lam\ , $$ 
which means that $e^\lam$ is an $\infty$-isotopy between $\at$ and ${\beta}$. One concludes with  
the bijective property of the pre-Lie exponential map of Lemma~\ref{lemma:ExponentialII}. 
\end{proof}

\begin{remarks}\leavevmode
\begin{itemize}
\item[$\diamond$]
This gauge equivalence first appeared in the paper \cite{Kadeishvili88} of T. Kadeshvili for the Koszul operad 
 encoding associative algebras.
 
\item[$\diamond$] The results of Theorem~\ref{thm:DeligneGroupoidII} give a conceptual framework to state in full generality the theory of Koszul hierarchy developed by M. Markl in \cite{Markl13, Markl14} for the operads $As$ and $Lie$. Let $\P$ be a Koszul operad, let $\mu : \P^!\to \End_A$ be an algebra structure over the Koszul dual operad, and let $\Delta$ be an degree $1$ square-zero operator on $A$. 
We work here with the cohomological degree convention and $s$ denotes an element of degree $1$.
Any basis of the Koszul dual operad $\P^!$ gives rise to a gauge group element defined by the composite
$$\Phi : \P^{\ac} \cong{\P^!}^*\otimes \End_{\KK s}\to {\P^!}\otimes \End_{\KK s}  \xrightarrow{\mu\otimes \id}
\End_A \otimes \End_{\KK s}\cong \End_{sA} \ .$$
Therefore, the gauge action of $\Phi$ on the trivial $\P_\infty$-algebra structure $\alpha_{\textrm{tr}}=(sA, \Delta)$ produces a $\P_\infty$-algebra structure $(\Phi^{-1}\star \alpha_{\textrm{tr}})\cc \Phi$ on the suspended chain complex $sA$. This construction answers the question raised in Section~$2.6$ of \cite{Markl13}. For the operad $\P=Lie$ of Lie algebras, its Koszul dual is the operad $\P^!=Com$ of commutative algebras. Using the canonical basis of this one-dimensional $\Sy$-module in any arity, one recovers the Koszul hierarchy \cite{Koszul85} of a commutative algebra, which forms an $L_\infty$-algebra, cf. \cite[Example~$2.8$]{Markl13}. 
In the same way, the nonsymmetric operad $As$ of associative algebras is Koszul auto-dual $As^!=As$ and one-dimensional in each arity.  Using its canonical basis, one recovers the B\"orjeson hierarchy \cite{Borjeson13} of an associative algebra, which forms an $A_\infty$-algebra, cf. \cite[Example~$2.4$]{Markl13}. Our results also explain why these new structures are homotopy trivial, that is they produce trivial structures on the cohomology level $H^\bullet(sA, \Delta)$, see Theorem~\ref{thm:GaugeTriv}.

\end{itemize}
\end{remarks}

Using this theorem, we can give yet another homotopy description of the Deligne groupoid. By definition, a  homotopy $\P$-algebra structure on a chain complex $(V,d)$ is a morphism of dg operads 
$$\P_\infty=\Omega \P^{\ac} \to \End_V\ .$$
Since, the source is cofibrant and the target is fibrant in the model category of dg operads \cite{Hinich97}, then  the homotopy relation is an equivalence relation, denoted $\sim_h$, which can be equivalently realized by any path or cylinder object. 

\begin{corollary}
The Deligne groupoid is isomorphic to the groupoid 
$$\boxed{\mathsf{Deligne}(\g_{\P, V})\cong 
\left(\mathrm{Hom}_{\mathsf{dg\ Op}}(\Omega \P^{\ac}, \End_V), \sim_h\right)}
\ ,$$
whose objects are morphisms of dg operads and whose morphisms are the homotopy equivalences. 
\end{corollary}

\begin{proof}
This result follows in a straightforward way from Theorem~\ref{thm:DeligneGroupoidII} and the cylinder object given by B. Fresse in \cite{Fresse09ter}. In loc. cit., the author proves that two such morphisms of dg operads are homotopy equivalent with respect to his cylinder object  if and only if the two homotopy $\P$-algebras are related by an $\infty$-isotopy. 
\end{proof}

\begin{remarks}\leavevmode
\begin{itemize}
\item[$\diamond$] In \cite{DolgushevRogers12}, V. Dolgushev and C. Rogers almost proved this result: they consider a path object type construction of dg operads but without proving the model category property. Then, they prove that the associated ``homotopy equivalence" is equivalent to the gauge equivalence (Theorem~$5.6$).
\item[$\diamond$] Using other  methods (coming from derived algebraic geometry), S. Yalin recently and independently proved a higher (simplicial) version of this result in \cite{Yalin14}. 
\end{itemize}
\end{remarks}

Homotopy $\P$-algebra structures can be equivalently described by the following three bijective sets: 
$$\mathrm{Hom}_{\mathsf{dg\ Op}}(\Omega \P^{\ac}, \End_V) \cong \MC(\g_{\P, V})\cong 
\mathrm{Codiff}\left(\P^{\ac}(V)\right)
  \ , $$
which is called the Rosetta Stone in \cite{LodayVallette12}. The various results of  this section show that these bijections of sets can be lifted to isomorphisms of groupoids: 
$$\left(\mathrm{Hom}_{\mathsf{dg\ Op}}(\Omega \P^{\ac}, \End_V), \sim_h\right) \cong \mathsf{Deligne}(\g_{\P, V})\cong (\P_\infty\textsf{-}\mathsf{Alg}, \infty\textsf{-}\mathsf{iso})
  \ .$$

\section{Homotopy properties of algebras over an operad}\label{sec:HoDiagOP}

In this section, we use the diagrammatic method, due to Rezk \cite{Rezk96}, to describe the homotopical properties of algebras over an operad. The first part, which recalls the general notion of cofibrant replacement and a proof of the homotopy transfer theorem from \cite{BergerMoerdijk03, Fresse10ter}, is not new. The second part deals with trivial homotopy transfer structure. This gives rise to a new general notion in any category theory, dual to cofibrant replacement which we call \textit{homotopy trivialization}.  These two parts will be refined in the next sections, using the preLie deformation calculus. 

\subsection{Categorical operadic interpretation} Let $\P$ be an operad and let $f : V \to H $ be a morphism of chain complexes. We consider two  $\P$-algebra structures on $V$ and on $H$ respectively, $\alpha : \P \to \End_V$ and $\beta : \P \to \End_H$.
Recall that the map $f$ is a morphism of $\P$-algebras if and only if the following diagram commutes  in the category of $\Sy$-modules 
$$ \xymatrix{\P   \ar[r]^(0.44)\beta \ar[d]_(0.45){\alpha}      &     {\End_H\ \ } \ar[d]^(0.45){f^*}    \\
\End_V  \ar[r]^(0.44){f_*}     &    {\End^V_HÊ\  ,}}  $$
where $\End^V_H(n):=\Hom(V^{\otimes n}, H)$. 
This is equivalent to requiring that the maps $\alpha$ and $\beta$ factor through the pullback 
$$ \xymatrix{\P  \ar@{..>}[rd]   \ar@/^1pc/[rrd]^(0.44)\beta \ar@/_1pc/[ddr]_(0.45){\alpha}  && \\
& \End_V \times_f \End_H    \ar[r]  \ar[d] \ar@{}[rd] | (0.3)\pullback &     {\End_H\ \ } \ar[d]^(0.45){f^*}    \\
& \End_V  \ar[r]^(0.44){f_*}     &    {\End^V_HÊ\  ,}}  $$
where 
$\End_V\times_f \End_H =\{ (g, h)\in \End_V\times \End_H\ | \ f\circ g(n)= h(n)\circ f^{\otimes n}\}$.

\subsection{Homotopy transfer theorem} Suppose now that we are only given a $\P$-algebra structure $\alpha$ on $V$ and that 
$f : V \qi H $ is a quasi-isomorphism of chain complexes. One of the main application of this section is given when $H$ is equal to the homology $H(V)$ of $V$; in this case, $f$ is a cofibration, assumption that we will keep in this section (only). 

The homotopy transfer theorem (HTT) gives an answer to the following question: what kind of algebraic structure, related to $\P$, can be transferred from $V$ to $H$? 

\begin{definition}[Cofibrant replacement]
In any model category, a \emph{cofibrant replacement} $\P_\infty$ of an object  $\P$ is a factorisation of the initial map $\I \to \P$ into a cofibration followed by an acyclic fibration: 
$$\boxed{\xymatrix{\I\ \     \ar@{>->}[r]  &
\P_\infty    \ar@{->>}[r]^\sim           & 
\P    }}\ . $$
\end{definition}
So cofibrant replacements are unique up to homotopy. We consider here cofibrant replacements 
in the model category of dg operads of \cite{Hinich97}. For instance, when the operad $\P$ is Koszul, the Koszul resolution $\Omega \P^{\ac}$ is cofibrant. 

Cofibrant replacements give the kind of algebraic structure that one can homotopically transfer onto $H$. 

\begin{proposition}[\cite{BergerMoerdijk03}]\label{pro:HTTcat}
For any acyclic cofibration $f : H \stackrel{\sim}{\mono} V$ and for any $\P$-algebra structure $\alpha$ on $V$, there exists a homotopy equivalent $\P_\infty$-algebra structure $\hat{\alpha}$ on $V$ and a $\P_\infty$-algebra structure $\beta$ on $H$, such that the chain map $f : H \to V$ becomes a quasi-isomorphism of $\P_\infty$-algebras from  $(H, \beta)$ to $(V, \hat{\alpha})$. 
\end{proposition}

\begin{proof}
We reproduce here quickly  the main ingredient of the proof of \cite[Theorem 7.3 (2)]{Fresse10ter} since we will use it  in the sequel. It is based on the following diagram of operads
$$ \xymatrix@M=8pt{\P_\infty \ar[r]   \ar@{..>}[dr] \ar@{..>}[drr]^{\widetilde{\alpha}} \ar@{..>}@/_1pc/[ddr]_\beta & \P \ar@/^1pc/[rd]^\alpha & \\
 & \End_V \times_f \End_H    \ar[r]^(0.57){\sim}  \ar[d] \ar@{}[rd] | (0.3)\pullback &     {\End_V\ \ } \ar[d]^(0.45){f^*}    \\
& \End_H  \ar[r]^(0.44){f_*}     &    {\End^H_VÊ\  .}}  $$
Since the operad $\P_\infty$ is cofibrant, the 
map $g : \End_V \times_f \End_H \qi \End_V$, which is a weak equivalence between fibrant operads, induces 
a bijection between the homotopy classes of morphisms of operads $[\P_\infty, \End_V \times_f \End_H]\cong 
[\P_\infty, \End_V]$. So, we can  choose a morphism $\theta : \P_\infty \to \End_V \times_f \allowbreak \End_H$, such that $\widetilde{\alpha}:=g\, \theta$ is homotopy equivalent to $\alpha$. The composition of $\theta$ with the map 
$ \End_V \times_f \End_H \to \End_H$ gives the required morphism $\beta$. 
\end{proof}

In this case, the transferred $\P_\infty$-algebra structure $\beta$ on $H$ is related to the original one $\alpha$ on $V$ by a zig-zag of quasi-isomorphisms: 
$$
\xymatrix{
(H,\beta) \ar[r]^{\sim}_f& (V, \widetilde{\alpha}) & {\cdot} \ar[l]_(0.35){\sim} \ar[r]^(0.35){\sim}& (V, {\alpha}) \ .
} $$
When the operad $\P$ is Koszul, one can use  the cylinder object of \cite{Fresse09ter} to describe the homotopy equivalence of operads maps. This shows that the chain map $f$ extends to an $\infty$-quasi-isomorphism: 
$$\boxed{
\xymatrix@C=27pt{
(H,\beta) \ar@{~>}[r]^{\sim}_f& (V, {\alpha})} 
} \ .$$

\begin{definition}[Homotopy transferred structure]
For any quasi-isomorphism $f : H \qi V$, a \emph{homotopy transferred structure} is any $\P_\infty$-algebra structure on  $H$ for which there exists an extension of the chain map $f$ into an $\infty$-quasi-isomorphism. 
\end{definition}

\subsection{Homotopy trivialization} Suppose now that
a  homotopy transferred structure is trivial (strictly or in general up to homotopy, see Definition~\ref{def:hotriv}), what does it mean about the initial $\P$-algebra structure on $V$? Can one already read on that level some homotopy trivialization structure? 
To this extend, we now  define the general notion of \emph{homotopy trivialization}, which  is dual to the notion of cofibrant replacement. 

\begin{definition}[Homotopy trivialization]
In any model category, a \emph{homotopy trivialization} $\hoP$ of an object  $\P$ is a factorisation of the terminal map  $\P \to 0$ into a cofibration followed by an acyclic fibration: 
$$\boxed{\xymatrix{\P\ \     \ar@{>->}[r]  &
\hoP    \ar@{->>}[r]^(0.6)\sim           & 
0}}\ . $$
\end{definition}
 Homotopy trivializations are well defined up to homotopy and this notion commutes with the notion of cofibrant replacement up to homotopy. 
 
\begin{lemma}\label{lem:ResHoTriv}
In any model category, a cofibrant replacement of a homotopy trivialization is homotopy equivalent to a homotopy trivialization of a cofibrant replacement: 
$$(\hoP)_\infty\simeq \mathrm{ho}(\P_\infty) \ .$$
\end{lemma}

\begin{proof}
We apply \cite[Proposition~$A.2.3.1$]{Lurie09} to lift the following 
commutative diagram in the homotopy category 
$$\xymatrix@M=7pt{\P_\infty   \ar[r]^\cong  \ar[d] & \ar[d]    \P \\ 
\mathrm{ho}(\P_\infty) \cong 0 \ar@{..>}[r]^0 &\ 0\cong \hoP} $$
to the following commutative diagram in  the actual category
$$\xymatrix@M=7pt{\P_\infty   \ar@{->>}[r]^\sim  \ar@{>->}[d] & \ar@{>->}[d]    \P \\ 
\mathrm{ho}(\P_\infty) \ar@{..>}[r]^\exists &\  \hoP\ .} $$
Since $\mathrm{ho}(\P_\infty)$ and $\hoP$ are homotopy trivial, this map is a weak equivalence. Finally, we apply the lifting property to the following diagram 
$$\xymatrix@M=7pt{\I \ar@{>->}[r] \ar@{>->}[d]&   (\hoP)_\infty  \ar@{->>}[d]^\sim \\
\mathrm{ho}(\P_\infty) \ar[r]^(0.52)\sim    \ar@{..>}[ur]^\sim& \hoP}$$
to get a weak equivalence between $(\hoP)_\infty$ and $\mathrm{ho}(\P_\infty)$.
\end{proof}

We consider now homotopy trivializations in the model category of dg operads of \cite{Hinich97}.
In this case, homotopy trivializations $\hoP$ of $\P$ are retracts of  a coproduct of $\P$ with a 
quasi-free operad $\TTT(Y)$  generated by a free $\Sy$-modules $Y$ endowed with a \emph{triangulation}, that is a filtration satisfying the following conditions: 
$$\boxed{\P\vee\TTT(Y_0\oplus Y_1\oplus \cdots)\quad \text{with}\quad  d(Y_n)\subset \P\vee\TTT(Y_0\oplus Y_1\oplus \cdots \oplus Y_{n-1})} \ .$$
We shall only be considering coproducts with quasi-free resolutions. In this case, the first ``syzygies'' $Y_0$ are homotopy trivializations of the generators of $\P$, the second ones $Y_1$ are homotopies for these homotopies, etc. 
The differential $d_\P+d_\chi$ is made up of the unique derivation $d_\chi$ which extends a  gluing map 
$\chi : Y \to \P\vee \TTT(Y)$. So it is often denoted by 
$$\boxed{\hoP = \P\vee_\chi \TTT(Y) := \left( \P\vee\TTT(Y), d_\P+d_\chi\right)}\ . $$

In the Koszul case, one can consider the following two homotopy trivializations of the operad $\P$ and of the operad $\P_\infty$ respectively. 
In the first case, the underlying $\Sy$-module is given by $\P\vee \TTT(\oPa)$. For the \emph{direct} one, the gluing map $\chi : \oPa \to \P\vee \TTT(\oPa)$ is 
$$\oPa \xrightarrow{\Delta_{(1)}} \oPa \circ_{(1)} \oPa \xrightarrow{\id \circ_{(1)} \kappa}  \oPa \circ_{(1)} \P \mono \P\vee \TTT(\oPa)\ ,$$
and for the \emph{opposite} one, the gluing map is 
$$\oPa \xrightarrow{\Delta} \oPa \circ \oPa \xrightarrow{\kappa \circ \id}  \P \circ \oPa \mono \P\vee \TTT(\oPa)\ .$$
In the case of the operad $\P_\infty$, the underlying $\Sy$-module is given by 
$\P_\infty\vee \TTT(\oPa)\cong \TTT(s^{-1}\oPa \oplus \oPa)$. The \emph{direct} one is produced by the gluing map $\chi : \oPa \to \TTT(s^{-1}\oPa \oplus \oPa)$ equal to  
$$\oPa \xrightarrow{\Delta_{(1)}} \oPa \circ_{(1)} \oPa \xrightarrow{\id \circ_{(1)} s^{-1}}  \oPa \circ_{(1)} s^{-1}\oPa \mono \TTT(s^{-1}\oPa \oplus \oPa)\ ,$$
and  the \emph{opposite} one is produced by the gluing map  
$$\oPa \xrightarrow{\Delta} \oPa \circ \oPa \xrightarrow{s^{-1} \circ \id}  s^{-1}\oPa \circ \oPa \mono \TTT(s^{-1}\oPa \oplus \oPa)\ .$$

\begin{lemma}Let $\P$ be a Koszul operad.  The above mentioned direct and the opposite operads are homotopy trivializations of the operad $\P$ and $\P_\infty$ respectively. 
\end{lemma}

\begin{proof}
Each of these four constructions provides us with a dg operad: it is enough to prove that  $(d_\chi)^2$, respectively 
$(d_{\P_\infty}+d_\chi)^2$, vanishes on $\oPa$. The result is given by a straightforward computation and the coassociative properties of the (infinitesimal) decomposition coproduct of the Koszul dual cooperad $\oPa$. 

The embedding $\P \mono \P\vee_\chi \TTT(\oPa)$, respectively $\P_\infty \mono \P_\infty \vee_\chi \TTT(\oPa)$, is a cofibration by \cite[Lemma~$6.7.1$]{Hinich97}.
The trivial map $\P\vee_\chi \TTT(\oPa) \epi 0$, respectively $\P_\infty\vee_\chi \TTT(\oPa) \epi 0$ is a fibration since it is surjective. 

It remains to prove that it is a quasi-isomorphism. 
In the case of the direct candidate for the operad $\P$, the underlying operad of $\P\vee_\chi \TTT(\oPa)$ is made up of trees with vertices alternately labelled by elements of $\P$ and $\TTT(\oPa)$. The differential only applies non-trivially to the upper vertices of $\TTT(\oPa)$, splitting the labeling element of $\oPa$ into two and composing the obtained top element with the above element in $\P$.  
So as a chain complex, the component $\P\vee_\chi \TTT(\oPa)$ without the trivial part $\I$, is isomorphic to a  direct sum of chain sub-complexes, which are made up of tensors of $\oPa$ with the augmentation part of the Koszul complex $\P^{\ac} \circ_\kappa \P$.  Since the operad $\P$ is Koszul, the Koszul complex is acyclic, which  proves that the homology of $\P\vee_\chi \TTT(\oPa)$ is trivial. 
The same arguments hold mutatis mutandis for the three other cases, where one has just to replace the Koszul complex 
$\P^{\ac} \circ_\kappa \P$ by, respectively, the acyclic Koszul complex $\P \circ_\kappa \P^{\ac}$, and the acyclic bar constructions 
$\P^{\ac} \circ_\iota \Omega \P^{\ac}$ and $\Omega \P^{\ac} \circ_\iota \P^{\ac}$.
\end{proof}

One can see that an algebra over the direct, respectively the opposite, homotopy trivialization amounts to a $\P$-algebra or a $\P_\infty$-algebra structure together with an $\infty$-isotopy to, respectively from, the trivial structure; that is an algebra structure together with a gauge group element whose action produces the trivial structure. Notice that here, we have an isomorphism $(\hoP)_\infty\cong \mathrm{ho}(\P_\infty)$ and not only a quasi-isomorphism.

\begin{remark}
Such homotopy trivializations applied to Batalin--Vilkovisky algebra structures were proved to play a key role in the description of Givental action on Cohomological Field Theories in \cite{DotsenkoShadrinVallette13preprint}.
\end{remark}

\begin{definition}[Homotopy trivial $\P$-algebra structure]\label{def:hotriv}
A $\P$-algebra structure $\P\to \End_H$ is \emph{homotopy trivial} if the associate morphism in the homotopy category is trivial. 
\end{definition}

\begin{proposition}\label{pro:HoTriv}
For any acyclic cofibration $f : H \stackrel{\sim}{\mono} V$ and for any $\P$-algebra structure on $V$, the homotopy transferred structure on $H$ of Proposition~\ref{pro:HTTcat} is homotopy trivial if and only if the $\P$-algebra structure on $V$ extends to a $\hoP$-algebra structure. 
\end{proposition}

\begin{proof}\leavevmode
\begin{itemize}
\item[$(\Longrightarrow)$] If the map $\beta$ is homotopy trivial, applying \cite[Proposition~$A.2.3.1$]{Lurie09} to the cofibration $\P_\infty \mono \allowbreak \mathrm{ho}(\P_\infty)$ shows that the map $\P_\infty \to \End_V \times_f \End_H$ factors through $\mathrm{ho}(\P_\infty)$. 
Lemma~\ref{lem:ResHoTriv} shows that the following diagram in the homotopy category is commutative 
$$\xymatrix@M=7pt{
\P_\infty \ar@{<->}[rr]^\cong  \ar@{>->}[d] & & \P  \ar@/^1pc/[rdd]^{[\alpha]} \ar@{>->}[d] & \\
\mathrm{ho}(\P_\infty) \ar@{<->}[r]^\cong     \ar@/_1pc/[rdrr] & (\hoP)_\infty \ar@{<->}[r]^(0.56)\cong & \hoP & \\
&&& \End_V\ .
}$$
Finally, since the map $\P \mono \hoP$ is a cofibration, we can apply again \cite[Proposition~$A.2.3.1$]{Lurie09} to prove the existence of a map $\hoP \to \End_V$ which factors the initial $\P$-algebra structure $\alpha$ on $V$. 

\item[$(\Longleftarrow)$] In the other way round, if the initial $\P$-algebra structure extends to a $\hoP$-algebra structure, this means that the structure map $\alpha$ factors through $\hoP$, that is $\alpha : \P \mono \hoP \to \End_V$. By pulling back with $\mathrm{ho}(\P_\infty) \to \hoP$, the space $V$ acquires a $\mathrm{ho}(\P_\infty)$-algebra structure. We can apply the arguments of the proof of Proposition~\ref{pro:HTTcat}, to produce a map $\mathrm{ho}(\P_\infty) \to \End_V \times_f \End_H$ whose composite with $\End_V \times_f \End_H \to \End_V$ is homotopy equivalent to the previous one. By the commutative diagrams of the proof of Lemma~\ref{lem:ResHoTriv}, this produces a required homotopy equivalent $\P_\infty$-algebra structure  $\widetilde{\alpha}$. Therefore, the induced transferred $\P_\infty$-algebra structure on $H$ factors through $\mathrm{ho}(\P_\infty)$. Since the operad $\mathrm{ho}(\P_\infty)$ is homotopy trivial,  so is this $\P_\infty$-algebra structure.
\end{itemize}
\end{proof}

In the case where $H$ is the homology of the chain complex V with a choice of lifting of the homology classes $H \mono V$, this proposition shows, together with the above mentioned homotopy trivializations of $\P_\infty$, that the transferred structure on $H$ of Proposition~\ref{pro:HTTcat} is homotopy trivial if and only if the original structure on $V$ is gauge trivial. We will generalize this result in the next section and prove that it actually does not depend on either a choice of a lifting or  a transferred structure. 

\section{Homotopy trivialization}\label{sec:HoTriv}

Using the pre-Lie deformation theory, we show the equivalence between gauge trivial  and homotopy trivial $\P_\infty$-algebra structures. In this section, we work with a Koszul operad  $\P$. 

\begin{lemma}\label{lem:HoTrivEA}
Let $\alpha$ be a $\P_\infty$-algebra structure on a chain complex $(V,d)$. 
The following assertions are equivalent. 

\begin{enumerate}
\item There exist a deformation retract onto the homology $H(V)$ for which a homotopy transferred structure is trivial.

\item Every homotopy transferred structure onto the homology $H(V)$ is trivial. 

\item The $\P_\infty$-algebra structure $\alpha$ on $V$ is homotopy trivial. 
\end{enumerate}
\end{lemma}

\begin{proof}\leavevmode
\begin{itemize}
\item[$(1)\Rightarrow(2)$]
Assumption $(1)$ provides us with a  trivial transferred structure
$$(H(V), 0, 0) \stackrel{\sim}{\rightsquigarrow} (V, d, \alpha) \ . $$
Let 
$$(H(V), 0, \beta) \stackrel{\sim}{\rightsquigarrow} (V, d, \alpha) \ . $$
be any other transferred structure. Since any $\infty$-quasi-isomorphism admits a homotopy inverse \cite[Theorem~$10.4.4.$]{LodayVallette12}, there exists an $\infty$-isomorphism 
$$f : (H(V), 0, 0) \stackrel{\cong}{\rightsquigarrow} (H(V), 0, \beta) \ . $$
Using the weight grading decomposition, the equation satisfied by the $\infty$-isotopy $f$ becomes 
\begin{eqnarray*}
f\star \beta &=& (f_{(0)} + f_{(1)} + f_{(2)} + \cdots) \star (\beta_{(1)}+\beta_{(2)}+\cdots ) \\
&=& \underbrace{f_{(0)}\star\beta_{(1)}}_{(1)} + 
\underbrace{f_{(0)}\star\beta_{(2)}  +f_{(1)} \star  \beta_{(1)}}_{(2)}
+\underbrace{f_{(0)}\star\beta_{(3)}  +f_{(1)} \star  \beta_{(2)}  +f_{(2)} \star  \beta_{(1)}  }_{(2)}+ \cdots \\
&=& 0
\ .
\end{eqnarray*}
Using the fact that $f_{(0)}(\I)\in \mathrm{GL}(H(V))$, one can see that $\beta=0$, by induction. 

\item[$(2)\Rightarrow(1)$] This is trivial.

\item[$(1)\Rightarrow(3)$] In this case, the minimal model theorem \cite[Theorem~$10.4.3$]{LodayVallette12} states that the original $\P_\infty$-algebra structure $\alpha$ is $\infty$-isotopic to the trivial structure. This proves the homotopy triviality of the map $\alpha$ in the model category of operads by the cylinder construction of \cite{Fresse09ter}.

\item[$(3)\Rightarrow(2)$] If the map of operads $\alpha$ is homotopy trivial, then the $\P_\infty$-algebra $(V,0)$ is $\infty$-isotopic to $(V, \alpha)$. In this case, any quasi-isomorphism of chain complexes $H(V)\qi V$ is any $\infty$-quasi-isomorphism of (trivial) $\P_\infty$-algebras. The composite of the two maps concludes the proof
\end{itemize}
\end{proof}

In plain words, a $\P_\infty$-algebra structure is homotopy trivial if and only if it transfers uniformly, for any choice of quasi-isomorphism and any choice of transferred structure, to a trivial structure on its underlying homology. 

\begin{theorem}\label{thm:GaugeTriv}
A $\P_\infty$-algebra structure is homotopy trivial if and only if it is gauge trivial. 
\end{theorem}

\begin{proof}
Let us first prove that gauge trivial implies homotopy trivial. By Theorem~\ref{thm:DeligneGroupoidII}, being gauge trivial implies that there exists an $\infty$-isotopy 
$$ (V, d, 0) \stackrel{=}{\rightsquigarrow} (V,d, \alpha)\ . $$ 
For any deformation retract onto the homology $H(V)$, the chain map 
$$ (H(V), 0, 0) \stackrel{\sim}{\rightsquigarrow} (V,d, 0)\ . $$ 
is an $\infty$-quasi-isomorphism for the trivial structures. The composite of these two $\infty$-morphims 
show that the $\P_\infty$-algebra structure $\alpha$ is homotopy trivial. \\
Let us now prove that homotopy  trivial implies gauge trivial. Suppose that there exists a deformation retract onto the homology $H(V)$ with an associated transferred structure which is trivial. The deformation retract data provides us with an isomorphism of chain complexes $(V, d)\cong (H(V)\oplus K, d_K)$, where $(K, d_K)$ is acyclic. By \cite[Theorem~$10.4.3$]{LodayVallette12}, this chain map extends into an $\infty$-isomorphism 
$$(V, d, \alpha) \stackrel{\cong}{\rightsquigarrow} (H(V)\oplus K, d_K, 0)\ .$$
The inverse chain map provides us with an $\infty$-quasi-isomorphism 
$$(H(V)\oplus K, d_K, 0) \stackrel{\cong}{\rightsquigarrow}  (V, d, 0)\ .$$
between trivial $\P_\infty$-algebra structures. Since the composite of these two $\infty$-morphisms is an $\infty$-isotopy, then $\alpha$ is gauge trivial, which concludes the proof. 
\end{proof}

This theorem establishes a useful method to prove the uniform triviality for transferred homotopy structures on homology, without having to compute any transfer formula. This result says that, in order to prove the homotopy triviality of a $\P_\infty$-algebra structure, it is enough to find a series 
$f=1+f_{(1)}+f_{(2)}+\cdots$ in $\Hom_\Sy({\P}^{\ac}, \End_V)$
satisfying 
$$ f\star \delta   = (\delta + \alpha)\circledcirc f \ , $$ 
or equivalently, 
a series 
$\lambda=\lambda_{(1)} + \lambda_{(2)}+\cdots$ in $\Hom_\Sy({\overline{\P}}^{\ac}, \End_V)$
satisfying
$$ \left( e^\lambda \star \delta \right)\circledcirc e^{-\lambda} = \delta + \alpha\ .$$

This method was used in the  simplest kind of homotopy algebras, namely multicomplexes, in \cite{DotsenkoShadrinVallette15}, to prove the homotopy triviality of the Koszul's Batalin--Vilkovisky operator on the de Rham cohomology of Poisson manifolds. This allowed us  to define  new higher invariants parametrized by the cohomology of the moduli space of genus $0$ curves with marked points.

\section{Homotopy transfer theorem}\label{subsec:HTT}

In this section, we interpret the Homotopy Transfer Theorem of Proposition~\ref{pro:HTTcat} and \cite[Chapter~$10$]{LodayVallette12} in terms of gauge action in the pre-Lie deformation theory. Moreover, we provide complete formulae together with a conceptual explanation for their internal structure.   \\



Since we are working over a field, the datum of an acyclic cofibration, that is an injective quasi-isomorphism, is equivalent to the datum of a \emph{contraction} \cite{EilenbergMacLane53}:
\begin{eqnarray*}
&\xymatrix{     *{ \quad \ \  \quad (V, d_V)\ } \ar@(dl,ul)[]^{h}\ \ar@<0.5ex>[r]^{p} & *{\
(H,d_H)\quad \ \  \ \quad }  \ar@<0.5ex>[l]^{i}}&\\
& i p-\id_V =d_V  h+ h  d_V\ , \quad pi=\id_H\ , \quad 
h^2=0\ , \quad ph=0\ , \quad hi=0 \ .
\end{eqnarray*}
This datum is equivalent to a deformation retract datum satisfying the \textit{side conditions} (the three last conditions). 
Notice that such a datum is completely equivalent to the sole datum of a degree one linear operator $h$
on the chain complex $(V, d_V)$ satisfying $h^2=0$ and $hdh=h$,  see \cite{BarnesLambe91}. For simplicity, we will denote the projection onto the image of $H$ inside $V$ by $\pi:=ip$. We consider the following symmetric homotopies 
$$ h_n:=\frac{1}{n!}\sum_{\sigma \in \Sy_n}  \sum_{k=1}^n \left(
\id^{\otimes (k-1)} \otimes h \otimes \pi^{\otimes (n-k)}
\right)^\sigma$$
from $\pi^{\otimes n}$ to $\id^{\otimes n}$. \\

%


Let $\P$ be a Koszul operad. We  work inside the convolution algebra 
$$\a=\a_{\P,V}= \big(
\Hom_\Sy({\P}^{\ac}, \End_V),  (\partial_V)_*, \star, \circledcirc, {1}
\big)\ . $$
Recall that  Maurer--Cartan elements $\bar\alpha\in\MC(\g)$ are considered under the ``augmented'' form $\alpha:=\delta + \bar\alpha$ in this convolution algebra $\a$. The set of group-like elements is $\mathrm{G}=\{ 1 \}\times
\Hom_\Sy(\overline{\P}^{\ac}, \End_V)
$.\\

Let us introduce the following two degree $0$  operators acting on $\a$:
$$
\begin{array}{lcll}
\LL \ : & \mathrm{G} & \to & \a  \\
 & x& \mapsto &  h_*\ba \circledcirc x
\end{array} 
\quad\quad \& \quad \quad
\begin{array}{lcll}
\R \ : & \a & \to & \a  \\
 & x& \mapsto & - h^*(x\star \ba)\ ,
\end{array} 
$$
where $h_*\ba$ is equal to the composite
$$\P^{\ac} \xrightarrow{\ba} \End_V \xrightarrow{h_*}  \End_V $$
and where $h^*(x\star \ba)$ is equal to the composite
$$\P^{\ac}(n) \xrightarrow{(x\star \ba)(n)} \Hom(V^{\otimes n}, V)  \xrightarrow{h_n^*}  \Hom(V^{\otimes n}, V)  \ .$$

\begin{definition}[$\Phi$-kernel and $\Psi$-kernel]\leavevmode 

\begin{itemize}
\item[$\diamond$] The \emph{$\Phi$-kernel} is the unique group-like element  $\Phi \in \mathrm{G}$ solution to the fixed point equation 
$$\Phi=1+\LL(\Phi) \ . $$
 It is  equal to 
$$\boxed{\Phi=(1-h_*\ba)^{\circledcirc -1}=e^{-\ln(1-h_*\ba)}=e^{-\Omega(-h_*\ba)}=\sum_{t\in \mathsf{RT}} 
\textstyle{\frac{1}{|\mathrm{Aut}\,  t|}}\, t(h\ba)}  \ . $$

\item[$\diamond$]  The \emph{$\Psi$-kernel} is the unique group-like element  $\Psi \in \mathrm{G}$ solution to the fixed point equation 
$$\Psi=1+\R(\Psi) \ . $$
 It is explicitly equal to 
$$\boxed{\Psi=1 +\R(1) + \R^2(1) +\R^3(1) + \cdots = 
1 - h^*\ba + h^*(h^*\ba \star \ba) - h^*(h^*(h^*\ba \star \ba) \star \ba) +\cdots }\  . $$
\end{itemize}
\end{definition}

The $\Phi$-kernel and the $\Psi$-kernel induce the following two operators acting on the Maurer--Cartan variety:
$$
\begin{array}{cll}
 \MC(\g) & \to & \MC(\g)   \\
  {\alpha}   & \mapsto & \hat{\alpha} :=\left( \Phi^{-1} \star {\alpha} \right) \cc \Phi 
\end{array} 
\quad  \quad \quad 
\begin{array}{cll}
 \MC(\g) & \to & \MC(\g)   \\
 {\alpha}  & \mapsto & \check{\alpha} :=\left( \Psi \star {\alpha} \right) \cc \Psi^{-1} Ê\ .
\end{array} 
$$

 Any endomorphism of $V$, like $h$ or $\pi$ for instance, can be viewed as an element of the convolution algebra $\a$ defined by 
$\I \mapsto h $, and denoted with the same notation. Therefore, pushing forward by an endomorphism is equal to 
$h_*\ba=h \star \ba = h\cc \ba$, which we simply denote by $h\ba$. 
In the same way, we will denote, by a slight abuse of notations, the pullbacks $i^*(x)$ and pushforwards $p_*(x)$ simply by $xi$ and $px$ respectively.

\begin{proposition}\label{prop:MChatcheck}
Any Maurer--Cartan element $\bar\alpha \in \MC(\g)$ satisfies 
$$\boxed{\hat{\alpha}=\delta  + \pi \cc \bar\alpha \cc \Phi}   \quad \textrm{and} \quad 
 \boxed{\check{\alpha}=\delta +  \bar\alpha \cc \Phi\cc \pi}\ . $$
 \end{proposition}

\begin{proof}
Using the Maurer--Cartan equation satisfied by $\ba+\delta$ and the homotopy relation, we get:
\begin{eqnarray*}
\Phi^{-1} \star (\delta +\bar \alpha) &= & (1-h\ba) \star (\delta + \ba)= \delta+  \ba- h\ba \star \delta  - h \ba\star \ba  \\
&=&\delta + \ba  - h (\ba \star \delta+\ba\star \ba )= \delta + \ba  +h (\delta\star \ba)\\
&=& \delta+ \ba  + (hd_V)  \star \ba= \delta+  \ba  + (-d_Vh +\pi-\id_V)  \star \ba\\
&=&   \delta\star(1-h\ba)+\pi\, \ba   \ .
\end{eqnarray*}
Composing on the right by $\cc\,  \Phi$, we obtain, by left linearity of the circle product:
\begin{eqnarray*}
\hat{\at} =\left( \delta\star(1-h\ba)+ \pi\, \ba  \right) \cc \Phi = \delta+ \pi\,(\bar\alpha \cc \Phi) \ .
\end{eqnarray*}

The other relation is proved as follows. 
By the definition of  $\check{\alpha}$, we have to prove 
$$\left( \Psi \star {\alpha} \right) \cc \Psi^{-1}=
\delta + \bar\alpha \cc \Phi\cc \pi\ , $$
which is equivalent to 
$$\delta \cc \Psi - \Psi \star \delta =\Psi \star \ba - \ba \cc \Phi \cc \pi \cc \Psi\ . $$
By analogy with the linear operator $\R$, we consider the following two linear operators acting on $\a$
$$
\begin{array}{lcll}
\R_\pi \ : & \a & \to & \a  \\
 & x& \mapsto & (-1)^{|x|}\, (x\star \ba)\cc \pi \ 
\end{array} \quad\quad \& \quad \quad
\begin{array}{lcll}
\R_1 \ : & \a & \to & \a  \\
 & x& \mapsto & (-1)^{|x|}\, x\star \ba\ .
\end{array} 
$$

Using the Maurer--Cartan equation $\delta  \star \ba = -\ba \star \delta -\ba \star \ba$ and the homotopy relation 
$\allowbreak \delta\star h_n=\allowbreak \pi^{\otimes n } - \id^{\otimes n} - h_n \star \delta $ as rewriting rules, we get 
\begin{eqnarray*}
\delta \cc \Psi - \Psi \star \delta&=& \delta \star \Psi - \Psi \star \delta
=  \delta \star \sum_{k\ge 0} \R^k(1) - \sum_{k\ge 0} \R^k(1) \star \delta=   \sum_{k\ge 0}\left( \delta\star \R^k(1) - 
\R^k(1) \star \delta\right) \\
&=&\sum_{k\ge 1} \sum_{l=1}^{k}  -\R^{k-l+1}\big( \R_1\big(\R^{l-1}(1)\big)\big) + \R^{k-l}\big( \R_1\big(\R^{l-1}(1)\big)\big) - \R^{k-l}\big( \R_\pi \big(\R^{l-1}(1)\big)\big)
\\
&=&\sum_{k\ge 0}  \R_1\big(\R^{k}(1)\big)- \sum_{k\ge 0} \R^k \R_\pi \Big(\sum_{l\ge 0} \R^l\Big) (1)\\
&=& \Psi \star \ba - \sum_{k\ge 0} \R^k \big(\R_\pi(\Psi)\big)  \  .
\end{eqnarray*} 
So it remains to prove that 
$$\sum_{k\ge 0} \R^k \big(\R_\pi(\Psi)\big)=\ba \cc \Phi\cc \pi \cc \Psi\ ,$$
which is a direct corollary of the following lemma. 
\end{proof}

\begin{lemma}\label{lem:TechR}\leavevmode
\begin{enumerate}
\item \quad $\R_\pi(\Psi)=\ba \cc \Phi \cc \pi$,
\item \quad $\sum_{k\ge 0} \R^k (x \cc \pi)=x\cc \pi  \cc \Psi\ .$
\end{enumerate}
\end{lemma}

To prove it, we will need the following formulas satisfied by the homotopies $h_n$. 

\begin{lemma}\label{lem:hn}\leavevmode
\begin{enumerate}
\item \quad $
h_{k+1+l} (h^{\otimes k} \otimes \id\otimes \pi^{\otimes l})=\big(h_{k+1} (h^{\otimes k}\otimes \id)\big)\otimes  \pi^{\otimes l}= \frac{1}{k+1}h^{\otimes k+1}\otimes \pi^{\otimes l}$,

\item \quad $(h_p\otimes \id^{\otimes q}- \id^{\otimes p}\otimes h_q)h_{p+q}=h_p\otimes h_q$.
\end{enumerate}
\end{lemma} 

A somewhat indirect proof of the latter relation is also contained in \cite[Sections 4-5]{Berglund14}.

\begin{proof}[Proof of Lemma~\ref{lem:hn}]\leavevmode
\begin{enumerate}
\item Using the side conditions, the first one comes from 
\begin{eqnarray*}
h_{n+k}(\id^{\otimes n}\otimes \pi^{\otimes k})&=&\frac{1}{(n+k)!}\sum_{\sigma \in \Sy_{n+k}}  \sum_{l=1}^{n+k} \left(
\id^{\otimes (l-1)} \otimes h \otimes \pi^{\otimes (n+k-l)}
\right)^\sigma (\id^{\otimes n}\otimes \pi^{\otimes k})\\
&=&\frac{1}{(n+k)!}\sum_{\sigma \in \Sy_{n}}  \sum_{l=1}^{n} 
\frac{(n+k)!}{n!}
\left(
\id^{\otimes (l-1)} \otimes h \otimes \pi^{\otimes (n-l)}
\right)^\sigma \otimes  \pi^{\otimes k}\\
&=&h_n\otimes \pi^{\otimes k}
\end{eqnarray*}
and the second one comes from  
\begin{eqnarray*}
h_{n+1}(h^{\otimes n}\otimes \id)&=&\frac{1}{(n+1)!}\sum_{\sigma \in \Sy_{n+1}}  \sum_{l=1}^{n+1} \left(
\id^{\otimes (l-1)} \otimes h \otimes \pi^{\otimes (n+1-l)}
\right)^\sigma (h^{\otimes n}\otimes \id)\\
&=&\frac{1}{(n+1)!}\sum_{\sigma \in \Sy_{n+1} \atop \sigma(n+1)=n+1}  \left(
\id^{\otimes n} \otimes h \right)^\sigma (h^{\otimes n}\otimes \id)\\
&=&\frac{1}{n+1}h^{\otimes n+1}\ . 
\end{eqnarray*}

\item Let us denote $n=p+q$. 
We use two combinatorial identities:
\begin{equation}\label{eq:1st-combinatorial}
\binom{a+b+c+1}{a+b+1} = \sum_{i+j=c} \binom{a+i}{a} \binom{b+j}{b}
\end{equation}
and
\begin{equation}\label{eq:2nd-combinatorial}
\binom{a+b+c+d+2}{a+b+1} = \sum_{i+j=b} \binom{a+c+i+1}{c} \binom{j+d}{d} + \sum_{i+j=d} \binom{a+c+i+1}{a} \binom{j+b}{b},
\end{equation}
which have both straightforward combinatorial proofs. 

The definition of $h_n$ can be rewritten as follows:
\begin{equation}\label{eq:h-newdefinition}
h_n= \sum_{\substack{I\sqcup P\sqcup H = [n] \\ |H|=1}} \mathrm{id}^{\otimes I}  \pi^{\otimes P} h^{\otimes H} \cdot \frac{|I|!|P|!}{n!}\ ,
\end{equation}
where we mean that we put $\mathrm{id}$ (resp. $\pi$, $h$) on the places indexed by $I$ (resp. $P$, $H$).
This way we can consider $h_p\otimes h_q$ as a sum over $[p]=I\sqcup P\sqcup H$ and $[q]=I'\sqcup P'\sqcup H'$ (here $[q]$ refers to the last $q$ places, so $[p]\sqcup[q]=[n]$). If we assume that $|I|=a$, $|P|=b$, $|I'|=c$, $|P'|=d$, then the coefficient of 
$\mathrm{id}^{\otimes I}  \pi^{\otimes P} h^{\otimes H} \otimes \mathrm{id}^{\otimes I'}  \pi^{\otimes P'} h^{\otimes H'}$ is equal to
\begin{equation}\label{eq:full-coefficient}
\frac{a!b!c!d!}{(a+b+1)!(c+d+1)!}\ .
\end{equation}

Let us compute the coefficient of the same expression in $(h_p\otimes \id^{\otimes q})h_n$. Consider particular summands in $h_p$ and $h_n$ corresponding to the partitions $[p]=I''\sqcup P''\sqcup H''$ and 
$[n]=I'''\sqcup P'''\sqcup H'''$. These partitions give $\mathrm{id}^{\otimes I}  \pi^{\otimes P} h^{\otimes H} \otimes \mathrm{id}^{\otimes I'}  \pi^{\otimes P'} h^{\otimes H'}$ if and only if we have
\begin{itemize}
\item[$\diamond$]
$I'''\cap [q] = I'$, $P'''\cap [q] = P'$
$H'''\cap [q] = H'$ (the last condition follows that $H'''\cap [p]=\emptyset$),
\item[$\diamond$]
$I'''\supset I$, $I'''\supset H$, and $(P'''\cap [p]) \subset P$,
\item[$\diamond$] $I''\supset I$, $H''=H$, and $P''\subset P$ such that $P''\cup (P'''\cap [p]) = P$.
\end{itemize}
Let us assume that $|P'''\cap [p]| = b_2$, $b_1+b_2=b$, and $|P''|=b_1+i$, $i+j=b_2$. With this assumption, if we take the sum over all possible choices of $I'',P'',H''$ and 
$I''',P''',H'''$ satisfying the conditions above and using Equation~\eqref{eq:h-newdefinition}, we see that the coefficient of $\mathrm{id}^{\otimes I}  \pi^{\otimes P} h^{\otimes H} \otimes \mathrm{id}^{\otimes I'}  \pi^{\otimes P'} h^{\otimes H'}$ in $(h_p\otimes \id^{\otimes q})h_n$ is equal to
\begin{equation*}
\sum_{b_1+b_2=b} \binom{b}{b_2} \frac{(a+b_1+c+1)!(b_2+d)!}{(a+b+c+d+2)!}
\sum_{i+j=b_2} \binom{b_2}{j} \frac{(a+i)!(b_1+j)!}{(a+b+1)!}\ .
\end{equation*}
Using the combinatorial identity~\eqref{eq:1st-combinatorial}, we can rewrite this as 
\begin{equation}\label{eq:1st-summand}
\frac{a!b!c!d!}{(a+b+c+d+2)!}\sum_{b_1+b_2=b}  \frac{(a+b_1+c+1)!(b_2+d)!}{
	c!(a+b_1+1)!b_2!d!}\ .
\end{equation}

Then we can compute, in exactly the same way, the coefficient of $\mathrm{id}^{\otimes I}  \pi^{\otimes P} h^{\otimes H} \otimes \mathrm{id}^{\otimes I'}  \pi^{\otimes P'} h^{\otimes H'}$ in $-(\id^{\otimes p}\otimes h_q)h_n$. It is equal to
\begin{equation}\label{eq:2nd-summand}
\frac{a!b!c!d!}{(a+b+c+d+2)!}\sum_{d_1+d_2=d}  \frac{(a+d_1+c+1)!(b+d_2)!}{
	a!(c+d_1+1)!b!d_2!} \ .
\end{equation}
The sign is changed from minus to plus because of the Koszul sign, since $h$ is odd.

Now, the  combinatorial identity~\eqref{eq:2nd-combinatorial} ensures us that the sum of~\eqref{eq:1st-summand} and~\eqref{eq:2nd-summand} is equal to the coefficient~\eqref{eq:full-coefficient}, which   completes the proof.

\end{enumerate}
\end{proof}

\begin{proof}[Proof of Lemma~\ref{lem:TechR}]\leavevmode

\begin{enumerate}

\item By definition, we have
\begin{eqnarray*}
\R_\pi(\Psi)&=&(\Psi\star \ba )\cc \pi\\
&=&
\big(\ba - h^*\ba\star \ba  + h^*(h^*\ba \star \ba)\star \ba - h^*(h^*(h^*\ba \star \ba) \star \ba)\star \ba +\cdots\big)\cc \pi \\
&=& \ba\cc \pi  - (h^*\ba\star \ba)\cc \pi   + (h^*(h^*\ba \star \ba)\star \ba)\cc \pi  + \cdots \ . 
\end{eqnarray*}
An element of the convolution pre-Lie algebra $\a$ like $(h^*(h^*\ba \star \ba)\star \ba)\cc \pi$ when applied to an element of $\P^{\ac}$ gives rise to an operation from $V^{\otimes n}$ to $V$ equal to a linear combination of composites of the form 
$$   \alpha_3 \, h_3 \, (\id^{\otimes 2}\otimes \alpha_2) \, h_4\, (\alpha_3\otimes \id^{\otimes 3})\, \pi^{\otimes 6} =\vcenter{
\xymatrix@M=4pt@R=8pt@C=8pt{
\pi & \pi & \pi & \pi & \pi\ar@{-}[dd] & & \pi \ar@{-}[dd] & \\
*{}\ar@{..}[r] & *+[o][F-]{}\ar@{-}[u]\ar@{-}[ur]\ar@{-}[ul] \ar@{..}[rrrrr] \ar@{-}[ddd] & & & & &*{}  & \\
*{}\ar@{--}[rrrrrr] & *{\bullet}& & *{\bullet}& *{\bullet} & & *{\bullet}& h_4\ \\
*{}\ar@{..}[rrrrr] & & & & & *+[o][F-]{}\ar@{-}[d]\ar@{-}[ul]\ar@{-}[ur]  \ar@{..}[r]& *{} & \\
*{}\ar@{--}[rrrrrr]&*{\bullet} & & *{\bullet}& & *{\bullet} &*{}  & h_3\ .\\
& & & *+[o][F-]{}\ar@{-}[d]\ar@{-}[ull] \ar@{-}[uuuuu]  \ar@{-}[urr] & & & & \\
& & & & & & & }}
$$
Using Relation~$(1)$ of Lemma~\ref{lem:hn},  this kind of composites is equal, up to sign, to 
$$1 \times \frac12\   \alpha_3 \, (h\alpha_3 \otimes \id \otimes h\alpha_2)\, \pi^{\otimes 6}=\frac12
\vcenter{
\xymatrix@M=4pt@R=8pt@C=8pt{
\pi & \pi & \pi & \pi & \pi & & \pi  \\
 & *+[o][F-]{}\ar@{-}[u]\ar@{-}[ur]\ar@{-}[ul]  & & & & *+[o][F-]{}\ar@{-}[ul]\ar@{-}[ur] &  \\
& & & *+[o][F-]{}\ar@{-}[d]\ar@{-}[ull]_h \ar@{-}[uu]  \ar@{-}[urr]^h & & &  \\
& & & & & & }}.
$$
{The idea is now to use this relation to pull down all the $\pi$'s and to rewrite these level-tree composites by replacing all the level-wise homotopies $h_n$  by simple homotopies $h$ labeling the internal edges.} This process produces the following combinatorial coefficient. 
To any levelization $\lambda$ of a rooted tree  $t\in \mathsf{RT}$, this coefficient, called the
weight $\omega(\lambda)$, is equal to 
the product over the levels between two vertices of one over the number of internal edges crossing this level, see Appendix~\ref{App:Comb} for more details. 

Therefore, we get
\begin{eqnarray*}
\R_\pi(\Psi)&=&\left(\sum_{t\in \mathsf{RT}} \sum_{\scriptstyle \text{levelization}\atop \scriptstyle \lambda\ \text{of} \ t} \omega(\lambda)\,  t(\ba, h)\right) \cc \pi\ , 
\end{eqnarray*}
where $t(\ba, h)$ stands for the element in the convolution pre-Lie algebra $\a$, which is the image in $\a$ of the element in the free pre-Lie algebra on two generators defined by the underlying tree $t$ with the root labelled by $\ba$ and the other vertices labelled by $h\ba$. 
Using Proposition~\ref{prop:CombTrees} which asserts 
$$
\sum_{\scriptstyle \text{levelization}\atop \scriptstyle \lambda\ \text{of} \ t} \omega(\lambda)=\frac{1}{|\mathrm{Aut}\, t|}\ , $$
we conclude that 
$$
\R_\pi(\Psi)=\left(\sum_{t\in \mathsf{RT}} \frac{1}{|\mathrm{Aut}\, t|} \,  t(\ba, h)\right) \cc \pi
= (\ba \cc \Phi)\cc \pi\ .
$$

\item We use the same method as above. 
By definition, we have
\begin{eqnarray*}
\sum_{k\ge 0} \R^k (x \cc \pi)&=&x \cc \pi  - h^*((x \cc \pi)\star \ba)   + 
h^*( h^*((x \cc \pi)\star \ba)\star \ba)
 + \cdots \ . 
\end{eqnarray*}
Applied to an  element of $\P^{\ac}$, an element like $h^*( h^*((x \cc \pi)\star \ba)\star \ba)$ produces the following kind of composite 

$$  x_3 \, \pi^{\otimes 3}\, (\id^{\otimes 2}\otimes \alpha_2)\, h_4 \, (\alpha_3\otimes \id^{\otimes 3})\, h_6 =\vcenter
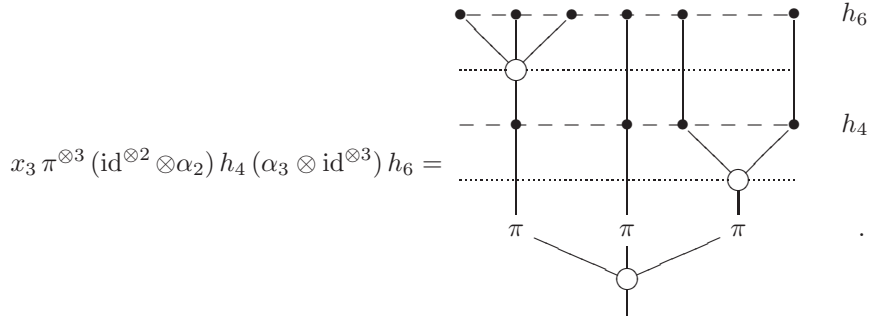
{
\xymatrix@M=4pt@R=8pt@C=10pt{
*{\bullet}\ar@{--}[rrrrrr] & *{\bullet} & *{\bullet} &  *{\bullet}&  *{\bullet} \ar@{-}[dd] & & *{\bullet}  \ar@{-}[dd] &h_6 \ \\
*{}\ar@{..}[r] & *+[o][F-]{}\ar@{-}[u]\ar@{-}[ur]\ar@{-}[ul] \ar@{..}[rrrrr] \ar@{-}[ddd] & & & & &*{}  & \\
*{}\ar@{--}[rrrrrr] & *{\bullet}*{}& & *{\bullet}*{}& *{\bullet} *{}& & *{\bullet}& h_4\ \\
*{}\ar@{..}[rrrrr] & & & & & *+[o][F-]{}\ar@{-}[d]\ar@{-}[ul]\ar@{-}[ur]  \ar@{..}[r]&*{}  & \\
&{\pi} & & \pi \ar@{-}[uuuu]  & & \pi & & \ .\\
& & & *+[o][F-]{}\ar@{-}[d]\ar@{-}[ull] \ar@{-}[u]  \ar@{-}[urr] & & & & \\
& & & & & & & }}
$$
Using the relation $(\pi^{\otimes k}\otimes \id^{\otimes l})\, h_{k+l} =h_k\otimes \pi^{\otimes l}$, we only deal with the induced $h_n$ labeling the edges above vertices different from the root. 
Since the abovementioned  composites produce operations from $V^{\otimes n}$, we can always put on the right-hand side the sub-tree above the root which contains the leaf corresponding to the last occurence of $V$ in $V^{\otimes n}$, and on the left-hand side, the rest of the sub-trees. 
Let us prove, by induction on $k=i+j\in\mathbb{N}$, the following statement:
\begin{itemize}
\item[$\left(\textsc{P}_{i,j}\right)$:]
for any underlying rooted trees $t$ with $k+1=i+j+1$ vertices, $i$ on the left-hand side above the root and $j$ on the right-hand side above the root, the element  $\R^k (x \cc \pi)\in \a$ producing composites equal to the sum over the levelizations of $t$, with $i$ vertices on the left-hand side above the root and $j$ vertices on the right-hand side above the root, 
 is equal to the element of $\a$ producing composites equal to the sum of levelizations of the sub-forest with $i$ vertices, denoted $\Lambda_i$ on the left-hand side above the root and composites corresponding to $\Psi_k$ on the right-hand side of the root. 
\end{itemize}
For any $k\ge 0$, the properties $\textsc{P}_{k,0}$ and $\textsc{P}_{0,k}$ are obvious. 
So, for instance, $\textsc{P}_{i,j}$ holds for $k=0$ and $k=1$. Suppose that $\textsc{P}_{i,j}$ holds for $i+j\leqslant k$ and let us prove it for $i+j= k+1$. Since $\textsc{P}_{k+1,0}$ and $\textsc{P}_{0,k+1}$ are proved, we can suppose $i, j\geqslant 1$. 
The element $\R^k (x \cc \pi)$, when applied to an element of $\P^{\ac}$ produces a linear combination of composites of the above type where the highest  appearance of $\ba$ is either on the right-hand side or on the left-hand side. Using the assumption hypotheses $\textsc{P}_{i,j-1}$ and $\textsc{P}_{i-1,j}$ respectively, the element 
$\R^k (x \cc \pi)$ produces composites of the form 
$$ 
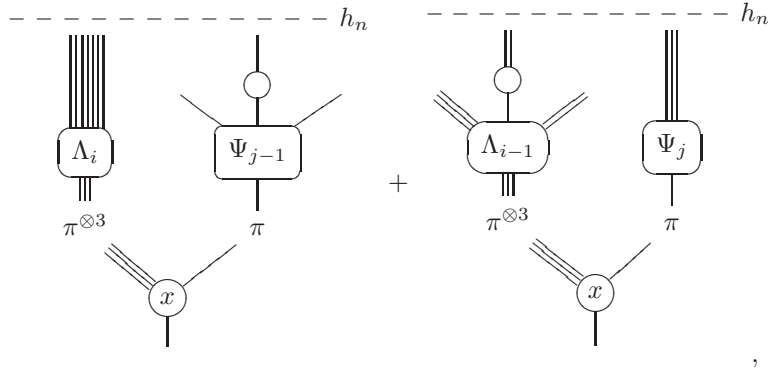
\vcenter{
\xymatrix@M=5pt@R=10pt@C=10pt{
*{} \ar@{--}[rrrr] && & & {h_n}   \\
 &    & & *+[o][F-]{}\ar@{-}[d] \ar@{-}[u]   & \\
&*+[F-:<5pt>]{\Lambda_{i}} \ar@<1ex>@2{-}[uu]\ar@<-1.1ex>@3{-}[uu]   \ar@2{-}[uu]& &*+[F-:<3pt>]{\Psi_{j-1}} \ar@{-}[ur]  \ar@{-}[ul] & \\
&\pi^{\otimes 3} \ar@3{-}[u] &&\pi \ar@{-}[u]& \\
&& *+[o][F-]{x} \ar@{-}[d] \ar@{-}[ur] \ar@3{-}[ul]  & & \\
&& & & 
}}+
\vcenter{
\xymatrix@M=5pt@R=10pt@C=10pt{
*{} \ar@{--}[rrrr] && & & {h_n}   \\
 & *+[o][F-]{}\ar@{-}[d] \ar@2{-}[u]     & & & \\
&*+[F-:<7pt>]{\Lambda_{i-1}} \ar@<0.5ex>@3{-}[ul]  \ar@<-0.5ex>@2{-}[ur]& &*+[F-:<5pt>]{\Psi_j} \ar@3{-}[uu] & \\
&\pi^{\otimes 3} \ar@3{-}[u] &&\pi \ar@{-}[u]& \\
&& *+[o][F-]{x} \ar@{-}[d] \ar@{-}[ur] \ar@3{-}[ul]  & & \\
&& & & ,
}}
$$
{where the dashed line represent the full level-wise homotopy $h_n$. }
Since  the highest part of 
$\Lambda_i$ is made of up an occurence of $\ba$ together with an $h_p$ and since 
 and since the highest part of 
$\Psi_j$ is made of up an occurence of $\ba$ together with an $h_q$, satisfying $p+q=n$, this is equal to 
$$ 
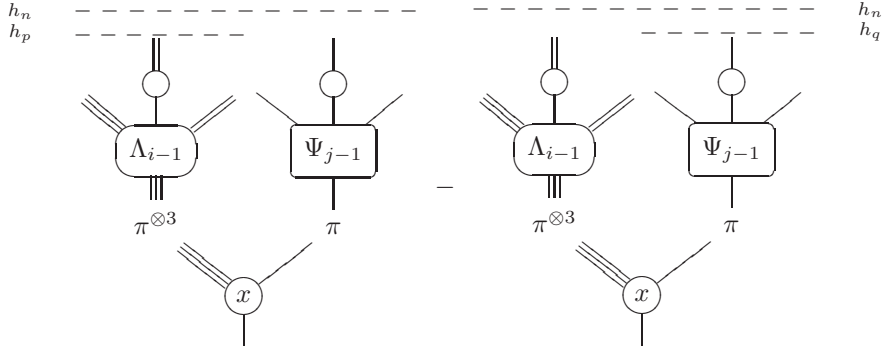
\vcenter{
\xymatrix@M=5pt@R=10pt@C=10pt{
*{h_n \atop h_p} &*{}\ar@<-1ex>@{--}[rr] *{} \ar@<1ex>@{--}[rrrr] && & & *{}    \\
& & *+[o][F-]{}\ar@{-}[d] \ar@2{-}[u]     & & *+[o][F-]{}\ar@{-}[d] \ar@{-}[u]  & \\
&&*+[F-:<7pt>]{\Lambda_{i-1}} \ar@<0.5ex>@3{-}[ul]  \ar@<-0.5ex>@2{-}[ur]& &*+[F-:<2pt>]{\Psi_{j-1}} \ar@{-}[ur]  \ar@{-}[ul] & \\
&&\pi^{\otimes 3} \ar@3{-}[u] &&\pi \ar@{-}[u]& \\
&&& *+[o][F-]{x} \ar@{-}[d] \ar@{-}[ur] \ar@3{-}[ul]  & & \\
&&& & & 
}}-
\vcenter{
\xymatrix@M=5pt@R=10pt@C=10pt{
*{} \ar@<1ex>@{--}[rrrr] && *{}\ar@<-1ex>@{--}[rr]& &*{} &   *{h_n \atop h_q}  \\
 & *+[o][F-]{}\ar@{-}[d] \ar@2{-}[u]     & & *+[o][F-]{}\ar@{-}[d] \ar@{-}[u]  && \\
&*+[F-:<7pt>]{\Lambda_{i-1}} \ar@<0.5ex>@3{-}[ul]  \ar@<-0.5ex>@2{-}[ur]& &*+[F-:<2pt>]{\Psi_{j-1}} \ar@{-}[ur]  \ar@{-}[ul] & &\\
&\pi^{\otimes 3} \ar@3{-}[u] &&\pi \ar@{-}[u]&& \\
&& *+[o][F-]{x} \ar@{-}[d] \ar@{-}[ur] \ar@3{-}[ul]  & & &\\
&& & & &.
}}
$$
Using Relation~$(2)$ of Lemma~\ref{lem:hn}, $(h_p\otimes \id^{\otimes q}- \id^{\otimes p}\otimes h_q)h_n=h_p\otimes h_q$, this is equal to 
$$ 
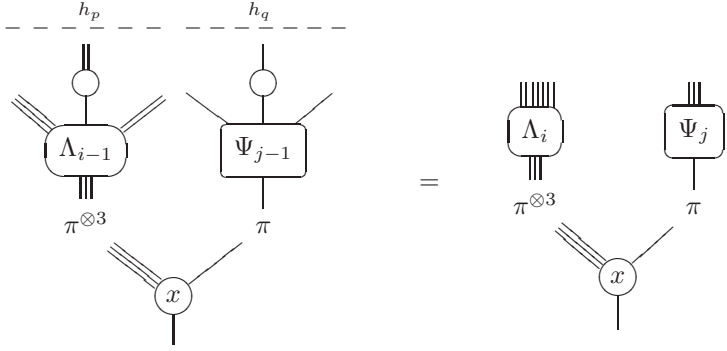
\vcenter{
\xymatrix@M=5pt@R=10pt@C=10pt{
*{}\ar@{--}[rr]^{h_p}  && *{\quad} \ar@{--}[rr]^{h_q}  & &  *{}   \\
 & *+[o][F-]{}\ar@{-}[d] \ar@2{-}[u]     & & *+[o][F-]{}\ar@{-}[d] \ar@{-}[u]  && \\
&*+[F-:<7pt>]{\Lambda_{i-1}} \ar@<0.5ex>@3{-}[ul]  \ar@<-0.5ex>@2{-}[ur]& &*+[F-:<2pt>]{\Psi_{j-1}} \ar@{-}[ur]  \ar@{-}[ul] & &\\
&\pi^{\otimes 3} \ar@3{-}[u] &&\pi \ar@{-}[u]&& \\
&& *+[o][F-]{x} \ar@{-}[d] \ar@{-}[ur] \ar@3{-}[ul]  & & &\\
&& & & &
}}=
\vcenter{
\xymatrix@M=5pt@R=10pt@C=10pt{
 &&  & &  \\
  &&  & &  \\
&*+[F-:<5pt>]{\Lambda_{i}} \ar@<1ex>@2{-}[u]\ar@<-1.1ex>@3{-}[u]   \ar@2{-}[u]& &*+[F-:<3pt>]{\Psi_{j}}  \ar@3{-}[u] & \\
&\pi^{\otimes 3} \ar@3{-}[u] &&\pi \ar@{-}[u]& \\
&& *+[o][F-]{x} \ar@{-}[d] \ar@{-}[ur] \ar@3{-}[ul]  & & \\
&& & &,  
}}
$$
which concludes the proof of the property $\textsc{P}_{i,j}$. 

In the end, we prove that 
$\sum_{k\ge 0} \R^k (x \cc \pi)= x\cc \pi \cc \Psi$, 
by iterating the property $\textsc{P}_{i,j}$ from right to left above the root.
\end{enumerate}
\end{proof}

In plain words, this proposition shows that the gauge twisted Maurer--Cartan element $\hat{\at}$ represents a new $\P_\infty$-algebra structure on $V$ whose operations have image in the sub-space $i(H)$. Respectively, the gauge twisted Maurer--Cartan element $\check{\alpha}$ represents a new $\P_\infty$-algebra structure on $V$ whose operations apply trivially outside to the sub-space $i(H)$. The next proposition makes this phenomenon even more precise. 

\begin{proposition}
The following assertions are equivalent: 
$$\ba =\pi \ba \quad \iff \quad \Phi=1 \quad \iff \quad \hat{\alpha}=\alpha\ .$$
This implies 
$$\boxed{\hat{\hat{\alpha}}=\hat{\alpha}}\ . $$
Dually, the following assertions are equivalent: 
$$\ba = \ba\pi  \quad \iff \quad \Psi=1 \quad \iff \quad \check{\alpha}=\alpha\ .$$
This implies 
$$\boxed{\check{\check{\alpha}}=\check{\alpha}} \ . $$
We also have 
$$ \boxed{\check{\hat{\alpha}}=\hat{\check{\alpha}}}\ .$$
\end{proposition}

\begin{proof}
If $\ba =\pi \ba$, then  $\Phi^{-1}=1-h\ba=1-h\pi\ba=1$, by the side conditions. So $\Phi=1$. If $\Phi=1$, then 
$\hat{\alpha}=(\Phi^{-1}\star\alpha)\cc \Phi=\alpha$. 
Finally, if $\hat{\alpha}=\alpha$, then 
$\pi \ba =\pi (\pi \cc \ba \cc \Phi) = \pi \cc \ba \cc \Phi=\ba$  by Proposition~\ref{prop:MChatcheck}.
Using  Proposition~\ref{prop:MChatcheck} again, we have $\pi \bar{\hat{\alpha}}=\bar{\hat{\alpha}}$, so 
$\hat{\hat{\alpha}}=\hat{\alpha}$.

The dual assertions are proved in the same way. 
If $\ba =\ba\pi$, then  $R(1)=0$, by the side conditions. So $\Psi=1$. If $\Psi=1$, then 
$\check{\alpha}=(\Psi\star\alpha)\cc \Psi^{-1}=\alpha$. 
Finally, if $\check{\alpha}=\alpha$, then 
$\ba \pi =(\ba \cc \Phi\cc \pi)\pi = \ba \cc \Phi\cc \pi=\ba$  by Proposition~\ref{prop:MChatcheck}.
Using Proposition~\ref{prop:MChatcheck} again, we have $\bar{\check{\alpha}}\pi=\bar{\check{\alpha}}$, so 
$\check{\check{\alpha}}=\check{\alpha}$.

Notice that the definitions of $\LL$, $\R$, $\Phi$, and $\Psi$ actually depend on the Maurer--Cartan element $\alpha$. To be precise, let us denote them using the index $\alpha$, like $\Phi_\alpha$ for instance. 
By Proposition~\ref{prop:MChatcheck}, we have 
\begin{eqnarray*}
\check{\hat{\alpha}}&=& \delta + \bar{\hat{\alpha}}\cc \Phi_{\hat{\alpha}}\cc \pi = \delta + \pi \cc \ba \cc \Phi\cc \pi \\
&=& 	\delta + \pi \cc (\ba \cc \Phi\cc \pi)\cc \Phi_{\check{\alpha}}	=	\delta + \pi  \cc \bar{\check{\alpha}}\cc \Phi_{\check{\alpha}}	\\
&=&\hat{\check{\alpha}}\ ,
\end{eqnarray*}
since $\pi \circ \Phi_{\check{\alpha}}=\pi$, by the side conditions. 

\end{proof}

This proposition shows that the choice of the gauge elements $\Phi$ and $\Psi$ are quite natural since the 
$\P_\infty$-algebra operations $\alpha$ on $V$ already have image in the sub-space $i(H)$ if and only if $\Phi$ is trivial. In this case, there is no need to perturb the original structure $\alpha$. Dually,  the 
$\P_\infty$-algebra operations $\alpha$ on $V$ are already restricted to the sub-space $i(H)$ if and only if $\Psi$ is trivial. So, there is again no need to perturb the original structure $\alpha$. These two restrictions by perturbation procedures are stable since $\hat{\hat{\alpha}}=\hat{\alpha}$, $\check{\check{\alpha}}=\check{\alpha}$, and $ \check{\hat{\alpha}}=\hat{\check{\alpha}}$.\\

The fact that the output (resp. the inputs) of $\hat{\alpha}$ (resp. $\check{\alpha}$) are restricted to $i(H)$ automatically ensures that $i$ (resp. $p$) becomes an $\infty$-quasi-isomorphism between the transferred structure $p\, \hat{\at}\, i$ and $\hat{\at}$ (resp. 
between $\check{\at}$ and the transferred structure $p\, \check{\at}\, i$). 

\begin{proposition}\label{prop:htt}\leavevmode
\begin{itemize}
\item[$\diamond$] The two elements 
$$\beta := p\, \hat{\at}\, i = p\, \check{\at}\,i $$ 
define the same Maurer-Cartan element in the convolution algebra $\a_{\P, H}$ of the space $H$. 

\item[$\diamond$] The following diagram of $\P_\infty$-algebras with $\infty$-morphisms is commutative: 
\begin{eqnarray*}
&&\xymatrix@C=50pt@R=25pt{
\check{\at} \ar@{~>}[r]^(0.48)p & p\, \check{\at}\, i  \\
\alpha  \ar@{~>}[u]^\Psi& \beta \ar@{=}[u]\\
 \hat{\at } \ar@{~>}[u]^\Phi & \ar@{~>}[l]^(0.52)i p\, \hat{\at}\,i \ar@{=}[u]}\\
&&\xymatrix@C=50pt@R=30pt{\text{in}\ V & \text{in}\ H}
\end{eqnarray*}
\end{itemize}
\end{proposition}

\begin{proof}\leavevmode
\begin{itemize}
\item[$\diamond$] Proposition~\ref{prop:MChatcheck} and the side conditions ensure that 
$$ p\, \hat{\at}\, i  = p(\delta  + \pi\,(\bar\alpha \cc \Phi))i= \delta + p(\bar\alpha \cc \Phi)i
= p(\delta + (\bar\alpha \cc \Phi)\cc \pi)i=p\, \check{\at}\, i  \ . $$

\item[$\diamond$] By Theorem~\ref{thm:DeligneGroupoidII}, the elements $\Phi$ and $\Psi$ are $\infty$-isotopies. 
The map $i$ defines an $\infty$-quasi-isomorphism between $p\, \hat{\at}\,i $ and $\hat{\at}$ since 
$$ i \star (p\, \hat{\at}\,i )  = \hat{\at} \cc i 
$$
by the side conditions and Proposition~\ref{prop:MChatcheck}.
In the same way, the map $p$ defines an $\infty$-quasi-isomorphism between $\check{\at} $ and $p\, \check{\at}\,i$ since 
$$  (p\, \check{\at}\,i )\cc p  = p\star \check{\at} 
$$
by the side conditions and Proposition~\ref{prop:MChatcheck}.

The side conditions, more precisely $h^2=0$, and the definitions of $\Phi=\sum_{t\in \mathsf{RT}} 
\frac{1}{|\mathrm{Aut}\, t|}\, t(h\ba) $ and $\Psi=1 - h^*\ba + h^*(h^*\ba \star \ba) - h^*(h^*(h^*\ba \star \ba) \star \ba) +\cdots  $ ensure that 
$\Psi \cc \Phi= \Psi + \Phi -1$. Therefore, the composite $p\cc \Psi\cc\Phi\cc i$ is equal to 
$$p\cc \Psi\cc\Phi\cc i =  p\cc(\Psi + \Phi -1)\cc i=pi=\id_H\, $$
again by the side conditions ($hi=0$ and $ph=0$). 

\end{itemize}
\end{proof}

\begin{theorem}
The above formulas define the following ingredients which solve  the Homotopy Transfer Theorem: 
\begin{itemize}
\item[$\diamond$] the transferred structure 
$$\boxed{\beta=\delta  + p\,(\bar\alpha \cc \Phi)\, i=
\sum_{t\in \mathsf{rRT}} 
\frac{1}{|\mathrm{Aut}\, t|}\ p\, t(\ba; h)\, i} \ , $$

\item[$\diamond$] the $\infty$-quasi-isomorphism 
$$\boxed{i_\infty= \Phi\,  i= \sum_{t\in \mathsf{rRT}} 
\frac{1}{|\mathrm{Aut}\, t|}\ h\, t(\ba; h)\, i} \ , $$

\item[$\diamond$] and the $\infty$-quasi-isomorphism 
$$\boxed{p_\infty= p\, \Psi}\ .$$

\end{itemize}
They satisfy  
\begin{eqnarray*}
 p_\infty \cc i_\infty=\id_H \ .
\end{eqnarray*}
\end{theorem}

\begin{proof}
This is a direct corollary of  Proposition~\ref{prop:htt}. \end{proof}

\begin{remarks} \leavevmode
\begin{itemize}
\item[$\diamond$] These formulas for $\beta$, $i_\infty$, and $p_\infty$ are the same as the ones given explicitly in \cite[Chapter~$10$]{LodayVallette12}. They also coincide with what can be extracted from the recursive formulas of  \cite{Berglund14}.

\item[$\diamond$] The next step would be to define a degree $1$ element $h_\infty\in \a$, which extends the homotopy $h$. Following the same ideas, one can consider $h_\infty:= h\, \Psi$, which is easily seen to satisfy the side conditions $h_\infty \cc h_\infty=0$, $p_\infty \cc h_\infty=0$, and $h_\infty\cc i_\infty=0$. But the homotopy relation between $h_\infty$, $i_\infty\cc p_\infty$, and $\id_V$ is more subtle and surely deserves  more study. 

\end{itemize}
\end{remarks}

This approach conceptually explains the discrepancy between these four formulae: the first two ones are of the same shape (trees) since they rely on the $\Phi$-kernel and the other two ones
are of the same shape (leveled trees) since they relie on the the $\Psi$-kernel. Such a phenomenon was first noticed by Markl \cite{Markl06} on the level of $A_\infty$-algebras. 
This discrepancy is present in the formulas of \cite[Chapter~$10$]{LodayVallette12} and  \cite{Berglund14}. 
See also \cite[Section~$4$]{LadaMarkl05}, where the role of the symmetric braces in the homotopy transfer theorem is guessed.

\appendix 

\section{Combinatorics of graphs }\label{App:Comb}

The purpose of this appendix is to prove a new combinatorial formula for rooted trees. This property plays a crucial role  in the proof of the main theorem of this paper, see  Section~\ref{subsec:HTT}.\\

\begin{definition}[Weight of a levelization]
To any levelization $\lambda$ of a rooted tree  $t\in \mathsf{RT}$, we associated its \textit{weight}
$\omega(\lambda)$ defined by the product over the levels between two vertices of one over the number of internal edges crossing this level. 
\end{definition}

\begin{example}
For instance, the weight of the following levelization 
$$\lambda=\vcenter{
\xymatrix@M=5pt@R=10pt@C=10pt{
*{}\ar@{..}[rr] & & *+[o][F-]{}\ar@{-}[dd]  & \ar@{}[d]^{1}_{\leftarrow}\\
*+[o][F-]{}\ar@{-}[ddr]\ar@{..}[rr] &&  *{} &\ar@{}[d]^{\frac12}_{\leftarrow} \\
*{}\ar@{..}[rr] &&  *+[o][F-]{}\ar@{-}[dl]&\ar@{}[d]^{\frac12}_{\leftarrow}\\
*{}\ar@{..}[r]& *+[o][F-]{} & *{}\ar@{..}[l] &
}}$$
is equal to 
$$\omega(\lambda)=1\times\frac12\times \frac12=\frac14\ .$$
\end{example}

\begin{proposition}\label{prop:CombTrees}
For any rooted tree $t\in \mathsf{RT}$, the following relation holds 
$$\boxed{ 
\sum_{\scriptstyle \text{levelization} \atop \scriptstyle \lambda\ \text{of} \ t} \omega(\lambda)=\frac{1}{|\mathrm{Aut}\, t|}}
\ .$$
\end{proposition}

\begin{proof}
Let us more generally prove this fact for a forest of trees. The notion of a levelization of a forest is the same as the one for  trees; the weight of a forest is defined similarly except that we multiply by the extra term equal to one over the number of trees of the forest:
$$\omega\left(\vcenter{
\xymatrix@M=5pt@R=10pt@C=10pt{
*+[o][F-]{}\ar@{-}[dddr] \ar@{..}[rrr] && & *{}  & \ar@{}[d]^{1}_{\leftarrow}\\
*{}\ar@{..}[rrr] && &*+[o][F-]{}\ar@{-}[ddd]  *{} &\ar@{}[d]^{\frac12}_{\leftarrow} \\
*{}\ar@{..}[rr] && *+[o][F-]{}\ar@{-}[dl]\ar@{..}[r]& *{} &\ar@{}[d]^{\frac13}_{\leftarrow}\\
*{}\ar@{..}[r]&*+[o][F-]{}\ar@{-}[d]*{}\ar@{..}[rr] & & *{} &\ar@{}[d]^{\frac12}_{\leftarrow}\\\
&&&&
}}\right)=1\times\frac12\times \frac13\times \frac12=\frac{1}{12}\ . $$
If a forest is reduced to a single tree, one recovers the above definition. 

Let us prove this result by induction on the number of vertices of the forest. When the forest $f$ has only one vertex, it is made up of one tree made up of  the sole  root vertex. In this case, the left-hand side and the  right-hand side of the formula are both equal to $1$. Suppose now the result true for forests with $n-1$ vertices. Let us prove it for any forest $f$ with $n$ vertices.

Let $f$ be a forest made up of $i_j$ trees of type $t_j$, for $j=1,\ldots, l$ and let $k=i_1+\cdots+i_l$. Given a levelization $\lambda$ of the forest $f$, the lowest vertex belongs to a tree of type $t_j$. If we cut such a levelization just above this vertex, we get a levelization $\lambda_j$ of an induced forest denoted $f_j$. In this case, one has to cut the tree $t_j$ just above its root, this gives rise to a forest made up of $c_{j,m}$ trees of type $t_m$, for $m\neq j$, and possibly trees of new types. We denote by $n_j$ the forest made up of these latter trees.  We consider 
$$\xi_j:=\prod_{m\neq j}\binom{i_{m}+c_{j,m}}{c_{j,m}}\ . $$
Under these notations, the number of automorphisms of the tree $t_j$ is equal to 
\begin{eqnarray*}
|\mathrm{Aut}\, t_j|=|\mathrm{Aut}\, n_j|
\prod_{m\neq j} \left(|\mathrm{Aut}\, t_m|^{c_{j,m}}\, c_{j,m}!\right)
\end{eqnarray*}
and the number of automorphisms
of the forest $f$ is equal to 
\begin{eqnarray*}
|\mathrm{Aut}\, f|=\prod_{j=1}^l\left( |\mathrm{Aut}\, t_j|^{i_j}\, i_j!\right) \ .
\end{eqnarray*}
Therefore,  the number of automorphisms of the forest $f_j$ is equal to 
\begin{eqnarray*}
|\mathrm{Aut}\, f_j|&=&|\mathrm{Aut}\, t_j|^{i_j-1}(i_j-1)!|\mathrm{Aut}\, n_j|
\prod_{m\neq j} \left(|\mathrm{Aut}\, t_m|^{i_m+c_{j,m}}(i_m+c_{j,m})!\right)\\
&=&|\mathrm{Aut}\, t| \frac{1}{|\mathrm{Aut}\, t_j|i_j}|\mathrm{Aut}\, n_j|
\prod_{m\neq j} \left(|\mathrm{Aut}\, t_m|^{c_{j,m}}\frac{(i_m+c_{j,m})!}{i_m!}\right)\\
&=&|\mathrm{Aut}\, t| \frac{1}{i_j}
\prod_{m\neq j} \left(\frac{(i_m+c_{j,m})!}{i_m!c_{j,m}!}\right)=|\mathrm{Aut}\, t| \frac{\xi_j}{i_j}
\ .
\end{eqnarray*}
The levelization $\lambda$ of the forest $f$ gives rise to the levelization $\lambda_j$ of the induced forest $f_j$. In the other way round, there are exactly $\xi_j$ ways to produce the levelization $\lambda$ of  $f$ from the levelization $\lambda_j$ of $f_j$: one has first to choose  $c_{j,m}$ trees of type $t_m$ from the   $i_{m}+c_{j,m}$ such trees inside $f_j$ and then to attach them, for any $m\neq j$, and all the trees of the forest $n_j$ onto a new root vertex, which is put on the new ground level of the levelization. 
In the end, the fact that  $\omega(\lambda)=\frac1k\omega(\lambda_j)$ and the induction hypothesis give
\begin{eqnarray*}
\sum_{\scriptstyle \text{levelization} \atop \scriptstyle \lambda\ \text{of} \ f} \omega(\lambda)&=&
\sum_{j=1}^l 
\sum_{\scriptstyle \text{levelization} \atop \scriptstyle \lambda_j\ \text{of} \ f_j} \xi_j \, \frac1k  \, \omega(\lambda_j)
=\frac1k\sum_{j=1}^l  
\frac{\xi_j}{|\mathrm{Aut}\, f_j|} =\frac1k\sum_{j=1}^l    \frac{i_j}{|\mathrm{Aut}\, f|}   \\&=&
\frac{1}{|\mathrm{Aut}\, f|} \ .
\end{eqnarray*}
\end{proof}

\bibliographystyle{amsalpha}

\end{document}